\documentclass[11pt]{amsart}
\usepackage{latexsym,amscd,amssymb, graphicx, color, amsthm, bm, amsmath, cancel, enumitem,young,chessboard}  
\usepackage{tikz}
\usepackage[margin=1in]{geometry}
\usepackage{hyperref}
\usepackage[all,cmtip]{xy}

\numberwithin{equation}{section}

\newtheorem{theorem}{Theorem}[section]
\newtheorem{proposition}[theorem]{Proposition}
\newtheorem{corollary}[theorem]{Corollary}
\newtheorem{lemma}[theorem]{Lemma}
\newtheorem{conjecture}[theorem]{Conjecture}
\newtheorem{observation}[theorem]{Observation}

\newtheorem{remark}[theorem]{Remark}
\newtheorem{definition}[theorem]{Definition}
\newtheorem*{fields1}{Fields Conjecture 1}
\newtheorem*{fields2}{Fields Conjecture 2}
\newtheorem*{fields3}{Fields Conjecture 3}

\newtheorem{mainthm}{Theorem}

\theoremstyle{definition}
\newtheorem{defn}[theorem]{Definition}

\newcommand{\Hilb}{{\mathrm{Hilb}}}

\newcommand{\tttt}{{\mathfrak{t}}}

\newcommand{\symm}{{\mathfrak{S}}}

\newcommand{\stair}{{\mathrm{st}}}

\newcommand{\Fl}{{\mathrm{Fl}}}

\newcommand{\zero}{{\mathbf{0}}}

\newcommand{\OP}{{\mathcal{OP}}}
\newcommand{\OMP}{{\mathcal{OMP}}}

\newcommand{\EEE}{{\mathcal{E}}}

\newcommand{\sign}{{\mathrm{sign}}}
\newcommand{\Stir}{{\mathrm{Stir}}}

\newcommand{\Ind}{{\mathrm{Ind}}}

\newcommand{\DDD}{{\mathfrak{D}}}

\newcommand{\AAA}{{\mathcal{A}}}
\newcommand{\BBB}{{\mathcal{B}}}

\newcommand{\HHH}{{\mathcal{H}}}

\newcommand{\LLL}{{\mathcal{L}}}

\newcommand{\III}{{\mathcal{I}}}

\newcommand{\PPP}{{\mathfrak{P}}}

\newcommand{\FF}{{\mathbb{F}}}
\newcommand{\CC}{{\mathbb{C}}}
\newcommand{\QQ}{{\mathbb{Q}}}
\newcommand{\RR}{{\mathbb{R}}}
\newcommand{\TT}{{\mathbb{T}}}
\newcommand{\ZZ}{{\mathbb{Z}}}
\newcommand{\WWW}{{\mathbb{W}}}

\newcommand{\xx}{{\mathbf{x}}}
\newcommand{\yy}{{\mathbf{y}}}

\newcommand{\sss}{{\mathbf{s}}}

\newcommand{\ann}{{\mathrm{ann}}}
\newcommand{\Frob}{{\mathrm{Frob}}}
\newcommand{\grFrob}{{\mathrm{grFrob}}}
\newcommand{\Top}{{\mathrm{top}}}

\newcommand{\Res}{{\mathrm{Res}}}
\newcommand{\Ker}{{\mathrm{Ker}}}

\newcommand{\inv}{{\mathrm{inv}}}
\newcommand{\maj}{{\mathrm{maj}}}
\newcommand{\dinv}{{\mathrm{dinv}}}
\newcommand{\minimaj}{{\mathrm{minimaj}}}
\newcommand{\des}{{\mathrm{des}}}
\newcommand{\SYT}{{\mathrm{SYT}}}

\newcommand{\Gale}{{\mathrm{Gale}}}

\begin{document}

\title[A proof of the Fields Conjectures]
{A proof of the Fields Conjectures}

\author[Satoshi Murai]{Satoshi Murai}
\address{Waseda University}
\email{s-murai@waseda.jp}

\author[Brendon Rhoades]{Brendon Rhoades}
\address{University of California, San Diego}
\email{bprhoades@ucsd.edu}

\author[Andy Wilson]{Andy Wilson}
\address{Kennesaw State University}
\email{awils342@kennesaw.edu}

\begin{abstract}
    The {\em superspace ring} of rank $n$ is the algebra $\Omega_n$ of differential forms on affine $n$-space. The algebra $\Omega_n$ is bigraded with respect to polynomial and exterior degree and carries a natural action of the symmetric group $\symm_n$. Modding out by $\symm_n$-invariants with vanishing constant term yields the {\em superspace coinvariant ring} $SR_n$. We prove that, as an ungraded $\symm_n$-module, the space $SR_n$ is isomorphic to the sign-twisted permutation action of $\symm_n$ on ordered set partitions of $\{1,\dots,n\}$. We refine this result by calculating the bigraded $\symm_n$-isomorphism type of $SR_n$. This proves the Fields Conjectures of N. Bergeron, L. Colmenarejo, S.-X. Li, J. Machacek, R. Sulzgruber, and M. Zabrocki as well as a related conjecture of V. Reiner.
\end{abstract}

\maketitle

\section{Introduction}
\label{sec:Introduction}

Let $\FF$ be a field of characteristic 0, let $\xx_n = (x_1,\dots,x_n)$ be a list of $n$ variables, and write $\FF[\xx_n] := \FF[x_1,\dots,x_n]$ for the polynomial ring with its usual grading. The symmetric group $\symm_n$ acts on $\FF[\xx_n]$ by subscript permutation and elements of the invariant subalgebra
\begin{equation}
    \FF[\xx_n]^{\symm_n} := \{ f \in \FF[\xx_n] \,:\, w \cdot f = f \text{ for all $w \in \symm_n$} \}
\end{equation}
are called {\em symmetric polynomials}. For $d > 0$ let $e_d \in \FF[\xx_n]^{\symm_n}$ be the {\em elementary symmetric polynomial}
$e_d := \sum_{1 \leq i_1 < \cdots < i_d \leq n} x_{i_1} \cdots x_{i_d}$
of degree $d$. 

Write $\FF[\xx_n]^{\symm_n}_+ \subseteq \FF[\xx_n]$ for the vector space of symmetric polynomials with vanishing constant term. The type A {\em coinvariant ideal} is
\begin{equation}
    I_n := (\FF[\xx_n]^{\symm_n}_+) = (e_1,\dots,e_n) \subseteq \FF[\xx_n]
\end{equation}
and the type A {\em coinvariant ring} is the corresponding quotient
\begin{equation}
    R_n := \FF[\xx_n]/I_n
\end{equation}
so that $R_n$ is both a graded ring and a graded $\symm_n$-module.

The module $R_n$ is among the most important representations in algebraic combinatorics, with algebraic properties governed by combinatorial properties of $\symm_n$. E. Artin proved \cite{Artin} that the set 
\begin{equation}
\label{eq:intro-artin-monomials}
    \{ x_1^{a_1} \cdots x_n^{a_n} \,:\, a_i < i \}
\end{equation}
of sub-staircase monomials descends to a vector space basis of $R_n$. Chevalley showed \cite{Chevalley} that 
\begin{equation}
\label{eq:chevalley-intro}
    R_n \cong_{\symm_n} \FF[\symm_n]
\end{equation}
as ungraded $\symm_n$-modules, so that $R_n$ is a graded refinement of the regular representation. The graded character of $R_n$ was calculated by Lusztig--Stanley \cite{Stanley} using standard Young tableaux.

The {\em superspace ring} of rank $n$ is the tensor product
\begin{equation}
    \Omega_n := \FF[x_1,\dots,x_n] \otimes_\FF \wedge \{ \theta_1,\dots,\theta_n\}
\end{equation}
of a rank $n$ symmetric algebra and a rank $n$ exterior algebra. The terminology arises from supersymmetry in physics, where commuting $x$-variables index bosons and anticommuting $\theta$-variables index fermions. The algebra $\Omega_n$ is bigraded; we call the $x$-grading  {\em bosonic} and the $\theta$-grading  {\em fermionic}. We may regard $\Omega_n$ as the algebra of regular differential forms on the affine space $\FF^n$. As such, it is equipped with an {\em Euler operator} 
\begin{equation}
    d: \Omega_n \longrightarrow \Omega_n, \quad \quad d f := \sum_{i=1}^n \frac{\partial f}{\partial x_i} \theta_i
\end{equation}
where the partial derivative $\partial/\partial x_i$ regards $\theta$-variables as constants.

The symmetric group $\symm_n$ acts on $\Omega_n$ diagonally, viz.
\begin{equation}
    w \cdot x_i = x_{w(i)}, \quad w \cdot \theta_i = \theta_{w(i)} \quad \quad \text{for $w \in \symm_n$ and $1 \leq i \leq n.$}
\end{equation}
Let $(\Omega_n)^{\symm_n}_+ \subseteq \Omega_n$ be the subspace of $\symm_n$-invariants with vanishing constant term. The {\em superspace coinvariant ideal} is the (two-sided) bihomogeneous ideal 
\begin{equation}
\label{eq:solomon-generators}
    SI_n := ((\Omega_n)^{\symm_n}_+) = (e_1,\dots,e_n,de_1,\dots,de_n) \subseteq \Omega_n
\end{equation}
where  the second equality is justified by a result of Solomon \cite{Solomon}. The {\em superspace coinvariant ring} is the bigraded $\symm_n$-module
\begin{equation}
    SR_n := \Omega_n / SI_n.
\end{equation}
Let $\mathrm{Mat}_n(\CC)$ be the affine space of $n \times n$ complex matrices, $\tttt \subseteq \mathrm{Mat}_n(\CC)$ be the diagonal matrices, and $\mathcal{N} \subseteq \mathrm{Mat}_n(\CC)$ be the nilpotent matrices. Thanks to \eqref{eq:solomon-generators}, when $\FF = \CC$ we may regard $SR_n$ as the algebra of K\"ahler differential forms on the scheme-theoretic intersection $\mathcal{N} \cap \tttt$.

The study of $SR_n$ was initiated by the Fields Institute Combinatorics Group\footnote{Nantel Bergeron, Laura Colmenarejo, Shu Xiao Li, John Machacek, Robin Sulzgruber, and Mike Zabrocki} in the late 2010's; see \cite{Zabrocki}. It was predicted that algebraic properties  of $SR_n$ are governed by combinatorial properties of the family $\OP_n$ of ordered set partitions of $[n] := \{1,\dots,n\}$. Deferring various definitions to Section~\ref{sec:Background}, the Fields Conjectures may be stated as follows. 

\begin{fields1}
    \label{conj:fields1} 
 The bigraded Hilbert series of $SR_n$ is 
    \begin{equation}
    \label{eq:hilbert-fields-intro}
        \Hilb(SR_n;q,z) = \sum_{k=1}^n z^{n-k} \cdot [k]!_q \cdot \Stir_q(n,k)
    \end{equation}
    where $\Stir_q(n,k)$ is a $q$-Stirling number, $q$ tracks bosonic degree, and $z$-tracks fermionic degree. In particular, we have $\dim_\FF SR_n = \# \OP_n$.
\end{fields1}

\begin{fields2}
We have an isomorphism of ungraded $\symm_n$-modules 
    \begin{equation}
        \label{eq:ungraded-fields-intro}
        SR_n \cong_{\symm_n} \FF[\OP_n] \otimes_\FF \sign
    \end{equation}
    where $\sign$ is the 1-dimensional sign representation of $\symm_n$.
\end{fields2}

\begin{fields3}
The bigraded Frobenius characteristic of $SR_n$ is given by 
    \begin{equation}
        \label{eq:graded-fields-intro}
        \grFrob(SR_n;q,z) = \sum_{k = 1}^n z^{n-k} \cdot C_{n,k}(\xx;q)
    \end{equation}
    where $C_{n,k}(\xx;q) = \Delta'_{e_{k-1}} e_n \mid_{t \to 0}$.
\end{fields3}

The expression $\Delta'_{e_{k-1}} e_n \mid_{t \to 0}$ involving the Macdonald eigenoperator $\Delta'_{e_{k-1}}$, together with other formulations of $C_{n,k}(\xx;q)$, will be recalled in Section~\ref{sec:Background}.

The Fields Conjectures inspired a flurry of research on $SR_n$  \cite{SS,Swanson,SW,SW2} and variants thereof \cite{Bergeron,KR,Lentfer,Lentfer2,RTW,RW2,RW3} but resisted proof. The inscrutable Gr\"obner theory of $SR_n$ rendered the Fields Conjectures impervious to direct attack. Recently, Rhoades and Wilson \cite{RW} used an algebraic argument to establish Fields Conjecture 1 and proved an `Operator Theorem' characterizing the inverse system $SI_n^\perp \subseteq \Omega_n$.  Angarone--Commins--Karn--Murai--Rhoades \cite{ACKMR} used  {\em Solomon--Terao algebras} (a recent invention \cite{AMMN} in the theory of hyperplane arrangements) to produce an explicit monomial basis of $SR_n$ generalizing Artin's basis of $R_n$ in \eqref{eq:intro-artin-monomials}. However, the methods of \cite{ACKMR,RW} did not respect the  $\symm_n$-action on $SR_n$, and the module structure of $SR_n$ predicted by the Fields Conjectures 2 and 3 remained open. Our main result is as follows.

\begin{mainthm}
\label{thm:intro-theoremA}
    The Fields Conjectures are true.
\end{mainthm}

As mentioned above, Fields Conjecture 1 was proven in previous work of Rhoades--Wilson \cite{RW}. Fields Conjecture 2 is Theorem~\ref{thm:ungraded-fields} while Fields Conjecture 3 is Theorem~\ref{thm:fields}. If we write $(SR_n)_{*,n-k}$ for the piece of $SR_n$ of fermionic degree $n-k$ and $\OP_{n,k}$ for the ordered set partitions of $[n]$ with $k$ blocks,  Fields Conjecture 3 implies
\begin{equation}
    (SR_n)_{*,n-k} \cong \FF[\OP_{n,k}] \otimes\sign
\end{equation}
as ungraded $\symm_n$-modules; see Corollary~\ref{cor:fermionic-piece}. We prove a conjecture of Reiner \cite{Reiner} which connects the restriction of $(SR_n)_{*,n-k}$ from $\symm_n$ to $\symm_{n-1}$ with the natural recursion on $\OP_{n,k}$ which erases $n$ (Corollary~\ref{cor:reiner}).

Borel proved \cite{Borel} that the cohomology of the variety $\Fl_n$ of complete flags in $\CC^n$ may be presented as $H^*(\Fl_n) = R_n$. One can ask for a geometric interpretation of $SR_n$; we give results in this direction. We use work of Reeder \cite{Reeder} to give a geometric interpretation of $SR_n$ in terms of the Springer resolution which extends to any Weyl group; see Proposition~\ref{prop:SR-geometric}. This leads to an interpretation of $SR_n$ as a geometric action on the top cohomology of certain fibers of the Springer morphism; see Corollary~\ref{cor:fiber}.

Replacing the superspace ring $\Omega_n$ by the polynomial ring $\FF[\xx_n,\yy_n]$ in $2n$ commuting variables $\xx_n = (x_1,\dots,x_n), \yy_n = (y_1,\dots,y_n)$, one still has a diagonal action of $\symm_n$, viz.
\begin{equation}
    w \cdot x_i := x_{w(i)}, \quad w \cdot y_i := y_{w(i)} \quad \quad \text{for $w \in \symm_n$ and $1 \leq i \leq n$.}
\end{equation}
Let $\FF[\xx_n,\yy_n]^{\symm_n}_+ \subseteq \FF[\xx_n,\yy_n]$ be the space of $\symm_n$-invariants with vanishing constant term. Setting $DI_n := (\FF[\xx_n,\yy_n]^{\symm_n}_+) \subseteq \FF[\xx_n,\yy_n]$, the study of the {\em diagonal coinvariant ring} $DR_n := \FF[\xx_n,\yy_n]/DI_n$ was initiated by Garsia and Haiman in the early 1990s; see \cite{HaimanQuotient}. Haiman used \cite{Haiman} the algebraic geometry of Hilbert schemes to calculate the bigraded character of $DR_n$ in terms of the $\nabla$-operator on symmetric functions. In 2018, Carlsson and Oblomkov \cite{CO} found an explicit monomial basis of $DR_n$ using the geometry of affine Springer fibers. Together with the basis result \cite{ACKMR} of Angarone et.\ al., Theorem~\ref{thm:intro-theoremA} gives the superspace analog of these results.

Our Theorem~\ref{thm:intro-theoremA} strategy is as follows. For a partition $\mu \vdash n$, let $\symm_\mu \subseteq \symm_n$ be the associated parabolic subgroup and let $\varepsilon_\mu \in \FF[\symm_n]$ be the group algebra element which antisymmetrizes over $\symm_\mu$. We calculate the dimensions of $\varepsilon_\mu \cdot SR_n$ for all $\mu \vdash n$; this is enough to pin down the character of $SR_n$. Using the basis result of Angarone et.\ al.\ \cite{ACKMR}, it is not difficult to  bound $\varepsilon_\mu \cdot SR_n$ from above. For the lower bound we switch to the inverse system $SH_n = SI_n^\perp$ to find sufficiently many linearly independent elements of $\varepsilon_\mu \cdot SH_n$. This requires more ingenuity. We give implicitly defined superspace operators $\DDD^\TT_\mu: \Omega_n \to \Omega_n$ depending on a `$\mu$-translation sequence' $\TT$. The application $\DDD^\TT_\mu(\delta_n)$ to the Vandermonde determinant $\delta_n \in \FF[\xx_n]$ results in elements of $\varepsilon_\mu \cdot SH_n$ whose fermionic leading terms have favorable properties. The various $\DDD^\TT_\mu(\delta_n)$ are used to bound $\varepsilon_\mu \cdot SH_n$ from below.

The approach in the previous paragraph bears some resemblance to the calculation of the bigraded Hilbert series of $SR_n$ performed in \cite{RW}. In that paper, the authors also study the harmonic space $SH_n$ and introduce $\DDD$-operators to bound $SH_n$ from below. Two main developments allow us to enhance this approach and find the  module structure of $SR_n$.
\begin{enumerate}
    \item The $\DDD$-operators used to bound $SH_n$ from below in the work of Rhoades--Wilson \cite[Sec. 4]{RW} have a complicated explicit definition related to flagged skew Schur polynomials. The intricacy of these operators made it unclear how to find a parabolic generalization. In contrast, we define our operators $\DDD^\TT_\mu$ {\em implicitly} using linear algebra. With this new paradigm, the relevant computations dramatically simplify and allow for parabolic generalization.
    \item Angarone--Commins--Karn--Murai--Rhoades \cite{ACKMR} exhibited a monomial basis $\AAA_n(J)$ of certain quotients $\FF[\xx_n]/(I_n:f_J)$ of the polynomial ring by colon ideals; see Theorem~\ref{thm:colon-ideal-basis}.  These bases will be used to understand the leading terms of $\DDD^\TT_\mu(\delta_n)$ and prove the desired lower bound on $\varepsilon_\mu \cdot SH_n$.

\end{enumerate}
The bases $\AAA_n(J)$ in (2) above were established via the recently defined {\em Solomon--Terao algebras} associated to hyperplane arrangements \cite{AMMN}. These algebras have also been used to present the cohomology of Hessenberg varieties \cite{AHMMS}. The Fields Conjectures are another application of these powerful algebras.

The rest of the paper is organized as follows. In {\bf Section~\ref{sec:Background}} we give background material on combinatorics, commutative algebra, the representation theory of $\symm_n$, and coinvariant theory. The short and expository {\bf Section~\ref{sec:RH}} surveys the most important results of \cite{ACKMR,RW} which will play a role in our proofs. {\bf Section~\ref{sec:Parabolic}} is the technical heart of the paper and culminates in an explicit basis (Lemma~\ref{lem:parabolic-dimension-equality}) of the bigraded vector space $\varepsilon_\mu \cdot SR_n$. In {\bf Section~\ref{sec:Main}} we prove Theorem~\ref{thm:intro-theoremA} and Reiner's Conjecture. In {\bf Section~\ref{sec:Springer}} we  relate $SR_n$ to Springer theory. In {\bf Section~\ref{sec:Conclusion}} we present a conjecture regarding other types.

\section{Background}
\label{sec:Background}

\subsection{Combinatorics}
Suppose $I = \{i_1 < \cdots < i_k \}$ and $J = \{j_1 < \cdots < j_k\}$ are two subsets of $[n]$ of the same size. The {\em Gale order} is defined by
\begin{equation}
    I \leq_\Gale J \quad \text{if and only if} \quad i_r \leq j_r \text{ for } r = 1,\dots,k.
\end{equation}

Let $n \geq 0$. A {\em partition} of $n$ is a sequence $\lambda = (\lambda_1 \geq \cdots \geq \lambda_p)$ of positive integers which sum to $n$. We write $\lambda \vdash n$ to indicate that $\lambda$ is a partition of $n$. We identify a partition $\lambda$ with its {\em (English) Young diagram} consisting of $\lambda_i$ left-justified boxes in row $i$. The Young diagram of $(4,2,2) \vdash 8$ is shown below. If $\lambda,\mu$ are two partitions, we write $\mu \subseteq \lambda$ for containment of their Young diagrams. The containment partial order on partitions $\lambda \subseteq (k^{n-k})$ is isomorphic to the Gale order on subsets of $[n]$ of size $k$. If $\lambda \vdash n$, we write $\lambda' \vdash n$ for the {\em conjugate} partition obtained by reflecting $\lambda$ across its main diagonal. For example, we have $(4,2,2)' = (3,3,1,1)$.

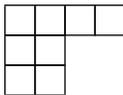
\begin{footnotesize}
    \[ \begin{young} & & & \cr & \cr & \cr\end{young}\]
\end{footnotesize}

If $\lambda \vdash n$ is a partition, a {\em $\lambda$-tableau} is a filling $T: \lambda \to \ZZ_{> 0}$ of the Young diagram of $\lambda$ with positive integers. We write $\lambda(T) := \lambda$ for the {\em shape} of $T$. A $\lambda$-tableau $T$ is {\em semistandard} if it is weakly increasing across rows and strictly increasing down columns. 
A semistandard tableau $T$ of shape $\lambda(T) = (4,2,2)$ is shown on the left below.

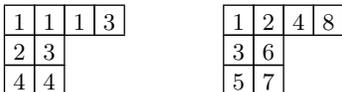
\begin{footnotesize}
    \[
    \begin{young}
        1 & 1 & 1 & 3 \cr
        2 & 3 \cr
        4 & 4
    \end{young}  \quad \quad \quad \quad
    \begin{young}
        1& 2& 4 & 8 \cr
        3& 6 \cr
        5& 7 \cr
    \end{young}
    \]
\end{footnotesize}

A semistandard tableau $T$ is {\em standard} if it consists of the numbers $1,2,\dots, n$, each occurring exactly once. A standard tableau of shape $(4,2,2)$ is shown on the right above.  If $T$ is a standard tableau with $n$ boxes, a number $1 \leq i \leq n-1$ is a {\em descent} of $T$ if $i$ appears in a strictly higher row than $i+1$ in $T$. The {\em descent number} $\des(T)$ is the number of descents of $T$ and the {\em major index} $\maj(T)$ is the sum of the descents of $T$.  The above standard tableau $T$ has descents $2,4,$ and $6$, so  $\des(T) = 3$ and $\maj(T) = 2+4+6=12$.

Write $\Lambda = \bigoplus_{n \geq 0} \Lambda_n$ for the ring of symmetric functions in an infinite variable set $\xx = (x_1,x_2,\dots)$ over a field containing $\QQ$ and the parameters $q,t,z$. We refer the reader to Haglund's book \cite{Haglund} for undefined symmetric function terminology.

Bases of $\Lambda_n$ are indexed by partitions $\lambda \vdash n$. We will make use of the {\em elementary basis} $\{e_\lambda\}$ and the {\em Schur basis} $\{s_\lambda\}$. If $\lambda = (\lambda_1,\dots,\lambda_p)$ is a partition, we have $e_\lambda := e_{\lambda_1} \cdots e_{\lambda_p}$ where $e_d := \sum_{1 \leq i_1 < \cdots < i_d} x_{i_1} \cdots x_{i_d}$ for all $d > 0$ and $e_0 := 1$. The Schur function is $s_\lambda := \sum_T \xx^T$ where the sum is over semistandard tableaux $T$ of shape $\lambda$ and $\xx^T = x_1^{c_1} x_2^{c_2} \cdots$ where $T$ has $c_i$ copies of $i$.

The bialternant formulation of the Schur basis will prove convenient. Let $N > 0$ and write $\FF[\symm_N]$ for the group algebra of $\symm_N$. This algebra acts on $\FF[\xx_N] = \FF[x_1,\dots,x_N]$. Let $\varepsilon_N := \sum_{w \in \symm_N} \sign(w) \cdot w \in \FF[\symm_N]$ be the antisymmetrizer. If $\lambda = (\lambda_1 \geq \cdots \geq \lambda_N)$ is a partition with $\leq N$ parts, one has 
\begin{equation}
\label{eq:bialternant-formula}
    s_\lambda(x_1,\dots,x_N) = \frac{\varepsilon_N \cdot (x_1^{\lambda_1 + N - 1} x_2^{\lambda_2 + N - 2} \cdots x_N^{\lambda_N})}{\varepsilon_N \cdot (x_1^{N-1} x_2^{N-2} \cdots x_N^0) }.
\end{equation}
The denominator of \eqref{eq:bialternant-formula} is the usual Vandermonde determinant
\begin{equation}
    \delta_N := \varepsilon_N \cdot (x_1^{N-1} x_2^{N-2} \cdots x_N^0) = \prod_{1 \leq i < j\leq N} (x_i  - x_j).
\end{equation}

The symmetric function $C_{n,k}(\xx;q) \in \Lambda_n$ arising in the Fields Conjecture 3 may be described in various ways. 
Recall the standard notation for $q$-numbers, $q$-factorials, $q$-binomial coefficients, and $q$-multinomial coefficients:
\begin{equation}
    [n]_q := \frac{q^n-1}{q-1} \quad [n]!_q := \prod_{i=1}^n[i]_q \quad {n \brack m}_q := \frac{[n]!_q}{[m]!_q \cdot [n-m]!_q} \quad {n \brack m_1,\dots,m_k}_q := \frac{[n]!_q}{[m_1]!_q \cdots [m_k]!_q}.
\end{equation}
We first give explicit expansions in symmetric function bases.
\begin{align}
    \label{eqn:C-classical}
    C_{n,k}(\xx;q) &:= \sum_{T \in \SYT(n)} q^{\maj(T) + {n-k \choose 2} - (n-k) \cdot \des(T)} {\des(T) \brack n-k}_q \cdot s_{\lambda(T)}  \\&= \sum_{\substack{\lambda \vdash n \\ \lambda'_1=k}} q^{\sum_i(i-1)(\lambda_i-1)} {k \brack m_1(\lambda),\dots,m_k(\lambda)}_q \cdot \omega \widetilde{H}_\lambda(\xx;q).
\end{align}
In the second line $\widetilde{H}_\lambda(\xx;q)$ is the {\em Hall--Littlewood function} attached to $\lambda$, we write $\omega: \Lambda \to \Lambda$ for the involution characterized by $\omega: s_\lambda \mapsto s_{\lambda'}$, and we let $m_i(\lambda)$ be the multiplicity of $i$ as a part of $\lambda$. See \cite[Thm. 6.14]{HRS} for more information.

An {\em ordered set partition} of $[n]$ is a list $\sigma = (B_1 \mid \cdots \mid B_k)$ of nonempty subsets of $[n]$ such that we have the disjoint union $[n] = B_1 \sqcup \cdots \sqcup B_k$. We write $\OP_n$ for the collection of ordered set partitions of $[n]$ and $\OP_{n,k}$ for the collection of ordered set partitions of $[n]$ with $k$ blocks. For example, we have
\[ ( 4, \, 5 \mid 2 \mid 1, \, 3, \, 6 ) \in \OP_{6,3}. \]Both $\OP_n$ and $\OP_{n,k}$ carry natural actions of the symmetric group $\symm_n$.

Let $\Stir(n,k)$ be the {\em Stirling number of the second kind} counting {\bf unordered} set partitions of $[n]$ into $k$ blocks. We have $\# \OP_{n,k} = k! \cdot \Stir(n,k)$ and $\# \OP_n = \sum_{k \geq 0} k! \cdot \Stir(n,k)$. The Hilbert series \eqref{eq:hilbert-fields-intro} in Fields Conjecture 1 involves a $q$-analog $\Stir_q(n,k)$ defined recursively by 
\begin{equation}
    \Stir_q(0,k) = \begin{cases}
        1 & k = 0, \\
        0 & k > 0,
    \end{cases} \quad \quad
    \Stir_q(n,k) = \Stir_q(n-1,k-1) + [k]_q \cdot \Stir_q(n-1,k).
\end{equation}
See \cite{Milne, Steingrimsson, WW} for combinatorial appearances of $q$-Stirling numbers.

An {\em ordered multiset partition} is a sequence $\mu = (S_1 \mid \cdots \mid S_k)$ of nonempty finite {\bf sets} of positive integers. We write $\OMP_{n,k}$ for the family of ordered multiset partitions  $\mu = (S_1 \mid \cdots \mid S_k)$ with $k$ blocks so that $\# S_1 + \cdots + \# S_k = n$. For example, we have
\[ (5  \mid 1, \, 2 , \, 4 \mid 1, \, 2 ) \in \OMP_{6,3}. \]
No repetition is allowed within the blocks of an ordered multiset partition. We write $\xx^\mu := x_1^{c_1} x_2^{c_2} \cdots $ where $c_i$ is the number of copies of $i$ in $\mu$; we have $\xx^\mu = x_1^2 x_2^2 x_4 x_5$ in the above example.
The symmetric function $C_{n,k}(\xx;q)$ can be described combinatorially using ordered multiset partitions. If $\mu$ is an ordered multiset partition, we have numbers $\inv(\mu),\maj(\mu),\dinv(\mu),$ and $\minimaj(\mu)$; see  \cite{Wilson} for their definitions. Combining results of \cite{HRW,Rhoades,Wilson} one has
\begin{multline}
    \label{eqn:C-combinatorial}
    C_{n,k}(\xx;q) = \\ \sum_{\mu \in \OMP_{n,k}} q^{\inv(\mu)} \xx^\mu = \sum_{\mu \in \OMP_{n,k}} q^{\maj(\mu)} \xx^\mu = 
    \sum_{\mu \in \OMP_{n,k}} q^{\dinv(\mu)} \xx^\mu = \sum_{\mu \in \OMP_{n,k}} q^{\minimaj(\mu)} \xx^\mu.
\end{multline}

There is a final way to write $C_{n,k}(\xx;q)$ in terms of Macdonald eigenoperators. For a partition $\lambda$, we write $\widetilde{H}_\lambda(\xx;q,t)$ for the modified Macdonald symmetric function. The set $\{ \widetilde{H}_\lambda \}$ forms a basis $\Lambda$. If $F \in \Lambda$ is any symmetric function, the {\em primed delta operator} is the Macdonald eigenoperator $\Delta'_F:\Lambda \to \Lambda$ characterized by
\begin{equation}
    \Delta'_F: \widetilde{H}_\lambda(\xx;q,t) \mapsto F(\dots,q^{(i-1)}t^{(j-1)},\dots) \times \widetilde{H}_\lambda(\xx;q,t)
\end{equation}
where the arguments of $F$ range over all matrix coordinates $(i,j) \neq (1,1)$ of boxes in $\lambda$. By results of \cite{GHRY, HRS2} (see also \cite{BHMPS, DM}) one has 
\begin{equation}
    \label{eqn:C-delta}
    C_{n,k}(\xx;q) = \Delta'_{e_{k-1}} e_n \mid_{t \to 0}.
\end{equation}

\subsection{(Super)Commutative algebra}
Recall that $\FF$ is a field of characteristic 0. If $V = \bigoplus_{i \geq 0} V_i$ is a graded $\FF$-vector space with each $V_i$ finite-dimensional, the {\em Hilbert series} of $V$ is 
\begin{equation}
    \Hilb(V;q) := \sum_{i \geq 0} \dim_\FF V_i \cdot q^i.
\end{equation}
More generally, if $V = \bigoplus_{i,j \geq 0} V_{i,j}$ is a bigraded vector space we have
\begin{equation}
    \Hilb(V;q,z) := \sum_{i,j \geq 0} \dim_\FF V_{i,j} \cdot q^i z^j.
\end{equation}
We use $q$ to track bosonic degree and $z$ to track fermionic degree. For a fixed fermionic degree $d$, we write
\begin{equation}
    V_{*,d} := \bigoplus_{i \geq 0} V_{i,d}
\end{equation}
for the fermionic degree part of $V$.

For $1 \leq i \leq n$ we have the partial derivative $\partial_i := \partial/\partial x_i: \FF[\xx_n] \to \FF[\xx_n]$. If $f = f(x_1,\dots,x_n) \in \FF[\xx_n]$,  let $\partial f: \FF[\xx_n] \to \FF[\xx_n]$ be the differential operator
$\partial f := f(\partial_1, \dots, \partial_n).$
We have a map $\odot: \FF[\xx_n] \times \FF[\xx_n] \to \FF[\xx_n]$ given by
$f \odot g := (\partial f)(g)$.
The $\odot$-action gives $\FF[\xx_n]$ the structure of an $\FF[\xx_n]$-module.

Let $I \subseteq \FF[\xx_n]$ be a homogeneous ideal. The {\em inverse system} (or {\em harmonic space}) associated to $I$ is
\begin{equation}
    I^\perp := \{ g \in \FF[\xx_n] \,:\, f \odot g = 0 \text{ for all $f \in I$}\}.
\end{equation}
It follows that $I^\perp$ is a graded subspace of $\FF[\xx_n]$ and one has 
\begin{equation}
    \Hilb(\FF[\xx_n]/I;q) = \Hilb(I^\perp;q).
\end{equation}
Furthermore, if the ideal $I$ is $\symm_n$-stable, then so is $I^\perp$ and one has 
\begin{equation}
    \FF[\xx_n]/I \cong I^\perp \quad \text{as graded $\symm_n$-modules.}
\end{equation}
Steinberg described the inverse system of the coinvariant ideal $I_n = (e_1,\dots,e_n) \subseteq \FF[\xx_n]$ as follows. Recall that $\delta_n = \prod_{1 \leq i < j \leq n} (x_i - x_j)$ is the Vandermonde determinant.

\begin{theorem}
    \label{thm:steinberg}
    {\em (Steinberg \cite{Steinberg})}  One has $\{ f \in \FF[\xx_n] \,:\, f \odot \delta_n = 0 \} =  I_n.$
\end{theorem}


As in the introduction, rank $n$ superspace is a tensor product $\Omega_n = \FF[x_1,\dots,x_n] \otimes_\FF \wedge \{ \theta_1,\dots,\theta_n\}$ of a rank $n$ symmetric algebra with a rank $n$ exterior algebra. This is a bigraded algebra and a bigraded $\symm_n$-module under the diagonal action
\begin{equation}
    w \cdot x_i := x_{w(i)}, \quad w \cdot \theta_i := \theta_{w(i)} \quad \quad \text{for $w \in \symm_n$ and $1 \leq i \leq n.$}
\end{equation}

Inverse systems extend  to $\Omega_n$. For $1 \leq i \leq n$, the partial derivative operator $\partial_i$ acts on the first tensor factor of $\Omega_n$. If $1 \leq i \leq n$ we have the {\em contraction operator} $\partial^\theta_i : \wedge \{ \theta_1,\dots,\theta_n\} \to \wedge \{ \theta_1,\dots,\theta_n\}$ characterized by linearity and
\begin{equation}
    \partial^\theta_i: \theta_{j_1} \cdots \theta_{j_r} \mapsto \begin{cases}
        (-1)^{s-1} \theta_{j_1} \cdots \widehat{\theta_{j_s}} \cdots \theta_{j_r} & \text{if $i = j_s$}, \\
        0 & \text{if $i \neq j_1,\dots,j_r$,}
    \end{cases}
\end{equation}
whenever $1 \leq j_1,\dots,j_r \leq n$ are distinct. The operator $\partial^\theta_i$ acts on $\Omega_n$ via the second tensor factor.

The operators $\partial_i=\partial/\partial x_i$ and $\partial_i^\theta$ satisfy the defining relations of $\Omega_n$. That is, we have
\[\partial_i \partial_j = \partial_j \partial_i, \quad \quad
    \partial_i \partial^\theta_j = \partial^\theta_j \partial_i, \quad \quad \partial^\theta_i \partial^\theta_j = - \partial^\theta_j \partial^\theta_i \quad \quad \text{for $1 \leq i,j \leq n$.} \] 
As such, any $f = f(x_1,\dots,x_n,\theta_1,\dots,\theta_n) \in \Omega_n$ gives rise to a well-defined operator $\partial f: \Omega_n \to \Omega_n$ given by
    $\partial f := f(\partial_1, \dots,\partial_n, \partial^\theta_1,\dots,\partial^\theta_n).$
This gives rise to an action $\odot: \Omega_n \times \Omega_n \to \Omega_n$ where $f \odot g := (\partial f)(g)$. 

It can be shown that if $I \subseteq \Omega_n$ is a bihomogeneous left ideal or right ideal, then $I$ is automatically a bihomogeneous two-sided ideal.\footnote{Since $(\Omega_n)_{i,j} = \FF[x_1,\dots,x_n]_i \otimes \wedge^j \{ \theta_1,\dots,\theta_n\}$, any two bihomogeneous elements of $\Omega_n$  commute or anticommute.}
The {\em inverse system} of a bihomogeneous ideal $I \subseteq \Omega_n$ is 
\begin{equation}
    I^\perp := \{ g \in \Omega_n \,:\, f \odot g = 0 \text{ for all $f \in I$} \}.
\end{equation}
Then $I^\perp \subseteq \Omega_n$ is a bigraded subspace with 
\begin{equation}
    \Hilb(\Omega_n/I;q,z) = \Hilb(I^\perp;q,z).
\end{equation}
Furthermore, if $I$ is $\symm_n$-stable, so is $I^\perp$, and we have \begin{equation}
    \Omega_n/I \cong I^\perp \quad \text{as bigraded $\symm_n$-modules.}
\end{equation}

\subsection{$\symm_n$-representation theory}
Since $\FF$ has characteristic 0, the irreducible representations of $\symm_n$ over $\FF$ are in one-to-one correspondence with partitions of $n$. If $\lambda \vdash n$ is a partition, we write $V^\lambda$ for the associated $\symm_n$-irreducible.

Let $V$ be a finite-dimensional $\symm_n$-module. There exist unique multiplicities $c_\lambda \geq 0$ such that $V \cong \bigoplus_{\lambda \vdash n} c_\lambda V^\lambda$. The {\em Frobenius characteristic} of $V$ is the symmetric function
\begin{equation}
    \Frob(V) := \sum_{\lambda \vdash n} c_\lambda  s_\lambda \in \Lambda_n.
\end{equation}
 More generally, if $V = \bigoplus_{i \geq 0} V_i$ or $V = \bigoplus_{i,j \geq 0} V_{i,j}$ is a graded or bigraded $\symm_n$-module with each piece finite-dimensional, we have the (bi)graded Frobenius characteristic
\begin{equation}
    \grFrob(V;q) := \sum_{i\geq 0} \Frob(V_i) \cdot q^i, \quad \quad \grFrob(V;q,z) := \sum_{i,j \geq 0} \Frob(V_{i,j}) \cdot q^i t^j.
\end{equation}

If $\mu = (\mu_1,\dots,\mu_p) \vdash n$ is a partition, we have the  {\em parabolic subgroup} $\symm_\mu = \symm_{\mu_1} \times \cdots \times \symm_{\mu_p} \subseteq \symm_n$. We write $\varepsilon_\mu \in \FF[\symm_n]$ for the group algebra element
\begin{equation}
    \varepsilon_\mu := \sum_{w \in \symm_\mu} \sign(w) \cdot w
\end{equation}
which antisymmetrizes over $\symm_\mu$. If $V$ is an $\symm_n$-module, then $\varepsilon_\mu \cdot V = \{ \varepsilon_\mu \cdot v \,:\, v \in V \}$ is an $\FF$-vector space. 

Let $V$ be an $\symm_n$-module and $\mu \vdash n$. The dimension of $\varepsilon_\mu \cdot V$ may be understood in terms of symmetric functions. Let $\langle -, - \rangle$ be the Hall inner product on $\Lambda$ with respect to which the Schur basis $\{s_\lambda\}$ is orthonormal. Given $F \in \Lambda$, we have a linear operator $F^\perp: \Lambda \to \Lambda$ characterized by
\begin{equation}
    \langle F^\perp G, H \rangle = \langle G, FH \rangle \quad \text{for all $G,H \in \Lambda.$}
\end{equation}
In other words, the operator $F^\perp$ is adjoint to multiplication by $F$. The connection between these operators and $\varepsilon_\mu$ is as follows.

\begin{lemma}
    \label{lem:module-antisymmetrization}
    For any $\symm_n$-module $V$ and any partition $\mu \vdash n$, one has \[e_\mu^\perp (\Frob(V)) = \dim_\FF (\varepsilon_\mu \cdot V).\]
\end{lemma}

Lemma~\ref{lem:module-antisymmetrization} is standard, but we include a proof for completeness.

\begin{proof}
Write $\mu = (\mu_1,\dots,\mu_p)$. By linearity, it suffices to prove the lemma when $V = V^\lambda$ is the irreducible corresponding to $\lambda \vdash n$. The Littlewood-Richardson Rule gives the $\symm_\mu$-module decomposition
\begin{equation}
    \Res^{\symm_n}_{\symm_{\mu}}(V^\lambda) \cong \bigoplus_{\nu^{(1)} \vdash\mu_1, \dots , \nu^{(p)}  \vdash \mu_p} c_{\nu^{(1)}, \dots, \nu^{(p)}}^\lambda \cdot V^{\nu^{(1)}} \otimes_\FF \cdots \otimes_\FF V^{\nu^{(p)}}.
\end{equation}
Here $c_{\nu^{(1)}, \dots, \nu^{(p)}}^\lambda$ is the {\em Littlewood-Richardson coefficient} giving the coefficient of $s_\lambda$ in the Schur expansion of $s_{\nu^{(1)}} \cdots s_{\nu^{(p)}}$. Writing $\sign$ for the sign character of $\symm_\mu$, we have
\begin{equation}
    \varepsilon_\mu \cdot V^\lambda = (\Res^{\symm_n}_{\symm_{\mu}}(V^\lambda))^\sign = c_{(1^{\mu_1}), \dots,(1^{\mu_p})}^\lambda \cdot V^{(1^{\mu_1})} \otimes_\FF \cdots \otimes_\FF V^{(1^{\mu_p})}.
\end{equation}
By the Dual Pieri Rule, we get $c_{(1^{\mu_1}), \dots,(1^{\mu_p})}^\lambda = K_{\lambda',\mu}$ is the {\em Kostka number} counting semistandard tableaux of shape $\lambda'$ with $\mu_i$ copies of $i$. Since $V^{(1^{\mu_1})} \otimes_\FF \cdots \otimes_\FF V^{(1^{\mu_p})}$ is 1-dimensional, we get 
\begin{equation}
    \dim_\FF (\varepsilon_\mu \cdot V^\lambda) = K_{\lambda',\mu} = \langle s_\lambda, e_\mu \rangle = e_\mu^\perp s_\lambda = e_\mu^\perp \Frob(V^\lambda)
\end{equation}
and we are done.
\end{proof}

\section{The spaces $SR_n$ and $SH_n$}
\label{sec:RH}

In this section we overview the results in \cite{ACKMR,RW} about $SR_n$ we will use in our proof of the Fields Conjecture.
Given a subset $J \subseteq [n]$, we define the {\em $J$-staircase} $\stair(J) = (\stair(J)_1, \dots, \stair(J)_n)$ recursively by
\begin{equation}
    \stair(J)_1 = \begin{cases}
        1 & 1 \notin J, \\
        0 & 1 \in J,
    \end{cases} \quad \quad
    \stair(J)_i = \begin{cases}
        \stair(J)_{i-1} + 1 & i \notin J, \\
        \stair(J)_{i-1} & i \in J.
    \end{cases}
\end{equation}
For example, if $n = 8$ and $J = \{3,6,7\}$ we have $\stair(J) = (1,2,2,3,4,4,4,5)$. Let $\AAA_n(J)$ be the set of monomials in $\FF[\xx_n]$ given by
\begin{equation}
    \AAA_n(J) := \{ x_1^{a_1} \cdots x_n^{a_n} \,:\, a_i < \stair(J)_i \text{ for all $i$}\}.
\end{equation}
We also set
\begin{equation}
    \AAA_n(J) \cdot \theta_J := \{ m \cdot \theta_J \,:\, m \in \AAA_n(J) \}
\end{equation}
where $\theta_J$ denotes the product $\prod_{j \in J} \theta_j$ in increasing subscript order. Finally, we let 
\begin{equation}
    \AAA_n := \bigsqcup_{J \subseteq [n]} \AAA_n(J) \cdot \theta_J.
\end{equation}
The following result was conjectured by Sagan and Swanson \cite{SS}.

\begin{theorem}
    \label{thm:artin-basis} {\em (Angarone--Commins--Karn--Murai--Rhoades \cite[Cor. 8.2]{ACKMR})} The set $\AAA_n$ descends to a vector space basis of $SR_n$.
\end{theorem}

In the case $n = 3$, this basis is as follows. We have 
\begin{center}
    \begin{tabular}{c | c | c | c | c}
        $J$  & $\varnothing$ & $3$ & $2$ & $23$\\ \hline
        $\stair(J)$ & $(1,2,3)$ & $(1,2,2)$ &  $(1,1,2)$ & $(1,1,1)$\\ \hline
        $\AAA_3(J)$ & $x_2 x_3^2, x_2 x_3, x_3^2, x_2, x_3,1$ & $x_2 x_3, x_2, x_3, 1$ &  $1,x_3$ & $1$ \\
    \end{tabular}
\end{center}
while $\AAA_3(J) = \varnothing$ when $1 \in J$. Thus, one obtains the basis
\[
\AAA_3 = \{x_2 x_3^2, x_2 x_3, x_3^2, x_2, x_3,1\} \sqcup \{x_2 x_3, x_2, x_3, 1\} \cdot \theta_3 \sqcup \{x_3,1\} \cdot \theta_2 \sqcup \{1\} \cdot \theta_2 \theta_3
\]
of $SR_3$.

The supercommutative algebra result in Theorem~\ref{thm:artin-basis} was proven in \cite{ACKMR} by establishing a family of $2^n$ results in commutative algebra. If $I \subseteq \FF[\xx_n]$ is an ideal and $f \in \FF[\xx_n]$ is a polynomial, recall that the {\em colon ideal} (or {\em ideal quotient}) $(I:f) \subseteq \FF[\xx_n]$ is given by
\begin{equation}
    (I:f) := \{ g \in \FF[\xx_n] \,:\, gf \in I \}.
\end{equation}
It is not hard to see that $(I:f)$ is an ideal which contains $I$. In particular, if $I_n \subseteq \FF[\xx_n]$ is the coinvariant ideal, we obtain an ideal $(I_n:f)$ which contains $I_n$.

The relevant polynomials for our colon ideals are indexed by subsets of $[n]$. If $J \subseteq [n]$, we define the polynomial $f_J \in \FF[\xx_n]$ by
\begin{equation}
    f_J := \prod_{j \in J} \left( x_j \times \prod_{j < i \leq n} (x_j - x_i) \right).
\end{equation}
The polynomials $f_J$ played a crucial role in the earlier works \cite{ACKMR, RW}, and will in this paper too.

\begin{theorem}
    \label{thm:colon-ideal-basis} {\em (Angarone--Commins--Karn--Murai--Rhoades \cite[Thm. 8.1]{ACKMR})} For any $J \subseteq [n]$, the set $\AAA_n(J)$ descends to a vector space basis of $\FF[\xx_n]/(I_n:f_J)$.
\end{theorem}

A transfer principle of Rhoades--Wilson \cite[Lem. 5.2]{RW} was used in \cite{ACKMR} to deduce Theorem~\ref{thm:artin-basis} from Theorem~\ref{thm:colon-ideal-basis}. When $J = \varnothing$, Theorem~\ref{thm:colon-ideal-basis} reduces to Artin's basis \eqref{eq:intro-artin-monomials} of $R_n$. For general $J \subseteq [n]$, the Gr\"obner theory of the colon ideal $(I_n:f_J)$ is badly behaved, and combinatorially showing that $\AAA_n(J)$ spans $\FF[\xx_n]/(I_n:f_J)$ seems to be  forbiddingly difficult. To get around this, Angarone et. al. used Solomon--Terao algebras of hyperplane arrangements.

The final result on $SR_n$ we need concerns its inverse system $SI_n^\perp \subseteq \Omega_n$. This space is often called the {\em superharmonic space} and denoted $SH_n := SI_n^\perp$. For $j \geq 1$, we let $d_j: \Omega_n \to \Omega_n$ be the {\em higher Euler operator}
\begin{equation}
    d_j(f) := \sum_{i=1}^n \frac{\partial^j f}{\partial x_j^j} \theta_i \quad \quad \text{for $f \in \Omega_n$.}
\end{equation}
The following result is the superspace analog of Steinberg's Theorem~\ref{thm:steinberg}. It was conjectured by Swanson and Wallach \cite{SW2}.

\begin{theorem}
    \label{thm:operator}
    {\em (Operator Theorem; Rhoades-Wilson \cite[Thm. 5.1]{RW})} The superharmonic space $SH_n = SI_n^\perp$ is the smallest linear subspace of $\Omega_n$ which $\dots$
    \begin{itemize}
        \item contains the Vandermonde determinant $\delta_n$,
        \item is closed under the higher Euler operators $d_1, d_2,\dots,d_{n-1}$, and
        \item is closed under the partial derivatives $\partial/\partial x_1,\dots,\partial/\partial x_n$.
    \end{itemize}
\end{theorem}

Haiman proved \cite{Haiman} an analogous Operator Theorem for $DI_n^\perp$.
    We will only use one direction of the Operator Theorem~\ref{thm:operator}: that $SH_n$ contains $\delta_n$, is closed under higher Euler operators, and is closed under the $\odot$-action of $\FF[\xx_n]$. This is the easier direction, and was proven already by Swanson and Wallach \cite{SW2}. Our methods dramatically simplify the proof of Theorem~\ref{thm:operator} given in \cite{RW}; see Remark~\ref{rmk:alternative-operator}.

\section{The space $\varepsilon_\mu \cdot SR_n$}
\label{sec:Parabolic}

We embark on the task of finding a basis of $\varepsilon_\mu \cdot SR_n$ for a partition $\mu \vdash n$. This result (Lemma~\ref{lem:parabolic-dimension-equality}) will allow us to apply Lemma~\ref{lem:module-antisymmetrization} to deduce the Fields Conjectures.

\subsection{Antisymmetrized monomials} Thanks to the monomial basis $\AAA_n$ of $SR_n$ in Theorem~\ref{thm:artin-basis}, it is easy to bound the dimension of $\varepsilon_\mu \cdot SR_n$ from above. We simply apply $\varepsilon_\mu$ to every element $m \in \AAA_n$ and remove obvious linear dependencies. That is, we remove instances where $\varepsilon_\mu \cdot m = 0$ and if $\varepsilon_\mu \cdot m = \pm \varepsilon_\mu \cdot m'$ for $m,m' \in \AAA_n$, we keep only one of $\varepsilon_\mu \cdot m,  \varepsilon_\mu \cdot m'$. The next pair of definitions set up the relevant combinatorics and will be used heavily throughout this section.

\begin{definition}
    \label{def:signed-partition} Let $n \geq 0$. A {\em signed partition} of $n$ is a pair $(\mu,\gamma)$ of sequences of the same length where $\mu \vdash n$ and $\zero \leq \gamma \leq \mu$. Here $\zero$ is a sequence of zeros and $\leq$ means componentwise inequality. If $(\mu,\gamma)$ is a signed partition of $n$, we define  $J(\mu,\gamma) \subseteq [n]$ by 
    \begin{equation}
        J(\mu,\gamma) := \bigsqcup_{j=1}^p \{ \mu_1 + \cdots + \mu_j - \gamma_j + 1, \dots , \mu_1 + \cdots + \mu_j - 1, \mu_1 + \cdots + \mu_j \}
    \end{equation}
    where $\mu = (\mu_1,\dots,\mu_p)$ and $\gamma = (\gamma_1,\dots,\gamma_p)$.
\end{definition}

An example signed partition of 8 is $(\mu,\gamma)$ where $\mu = (3,3,2)$ and $\gamma = (1,2,0)$. The partition $\mu$ subdivides the set $[8]$ into intervals $(1, \,2, \, 3 \mid 4, \, 5, \, 6 \mid 7,\, 8)$ of sizes $\mu_1,\mu_2,$ and $\mu_3$. The set $J(\mu,\gamma)$ is obtained by selecting the largest $\gamma_j$ entries from the $j^{th}$ interval; we have $J(\mu,\gamma) = \{3,5,6\}$ in our case. Also, we have $\stair(J(\mu,\gamma)) = (1,2,2 \mid 3,3,3 \mid 4 ,5 ).$ The bars in $\stair(J(\mu,\gamma))$ are placed after positions $\mu_1, \mu_1 + \mu_2,$ and so on to remind the reader of the role played by $\mu$. Our next definition attaches a set of monomials in $\FF[\xx_n]$ to any signed partition of $n$.

\begin{definition}
    \label{def:antisymmetrized-artin} Let $(\mu,\gamma)$ be a signed partition of $n$. We define $\AAA_n(\mu,\gamma)$ to be the following set of monomials in $\FF[\xx_n]$:
    \begin{equation}
        \AAA_n(\mu,\gamma) := \left\{ x_1^{a_1} \cdots x_n^{a_n} \,:\, \begin{array}{c} \text{$a_i < \stair(J(\mu,\gamma))_i$ \text{ for all $i$},} \\ \text{$a_{\mu_1 + \cdots + \mu_{j-1} + 1} < \cdots < a_{\mu_1 + \cdots + \mu_j - \gamma_j}$ for all $j$, and} \\
        \text{$a_{\mu_1 + \cdots + \mu_j - \gamma_j + 1} \leq \cdots \leq a_{\mu_1 + \cdots + \mu_j}$ for all $j$}\end{array} \right\}.
    \end{equation}
    We write $\varepsilon_\mu \cdot (\AAA_n(\mu,\gamma) \cdot \theta_{J(\mu,\gamma)})$ for the collection of superspace elements 
    \begin{equation}
        \varepsilon_\mu \cdot (\AAA_n(\mu,\gamma) \cdot \theta_{J(\mu,\gamma)}) := \{ \varepsilon_\mu \cdot (x_1^{a_1} \cdots x_n^{a_n} \cdot \theta_{J(\mu,\gamma)}) \,:\, x_1^{a_1} \cdots x_n^{a_n} \in \AAA_n(\mu,\gamma) \}.
    \end{equation}
\end{definition}

Continuing our example of $\mu = (3,3,2)$ and $\gamma = (1,2,0)$, we have $J(\mu,\gamma) = \{3,5,6\}$ and the $J(\mu,\gamma)$-staircase is $\stair(J(\mu,\gamma)) = (1,2,2\mid3,3,3\mid4,5)$. Definition~\ref{def:antisymmetrized-artin} gives
\begin{footnotesize}
\[ \AAA_n(\mu,\gamma) = \left\{ x_2 \cdot x_3, x_2  \right\} \otimes \left\{ 
\begin{array}{c} x_4^2 \cdot  x_5^2   x_6^2, \, x_4^2 \cdot x_5 x_6^2 , \, x_4^2 \cdot x_5 x_6, \, x_4^2 \cdot x_6^2, \, x_4^2 \cdot x_6, \, x_4^2, \\
x_4 \cdot  x_5^2   x_6^2, \, x_4 \cdot x_5 x_6^2 , \, x_4 \cdot x_5 x_6, \, x_4 \cdot x_6^2, \, x_4 \cdot x_6, \, x_4, \\
1 \cdot  x_5^2   x_6^2, \, 1 \cdot x_5 x_6^2 , \, 1 \cdot x_5 x_6, \, 1 \cdot x_6^2, \, 1 \cdot x_6, \, 1
\end{array} 
\right\} \otimes \left\{  \begin{array}{c} 
x_7^3 x_8^4, \, x_7^2 x_8^4, \, x_7 x_8^4, \\
 x_8^4, \, x_7^2 x_8^3, \, x_7 x_8^3, \\
x_8^3, \, x_7 x_8^2, \, x_8^2, \, x_8 
\end{array} \right\}. \]
\end{footnotesize}

\noindent
The tensor symbols correspond to the partition of $\{x_1,\dots,x_8\}$ induced by $\mu$ and mean to take all possible products of  monomials in the first, second, and third sets. The dots correspond to the factorization of an interval $\{x_{\mu_1 + \cdots + \mu_{j-1} + 1}, \dots, x_{\mu_1 + \cdots + \mu_{j-1} + \mu_j} \}$ into the two subintervals
$\{x_{\mu_1 + \cdots + \mu_{j-1} + 1}, \dots, x_{\mu_1 + \cdots +  \mu_j - \gamma_j} \}$ and 
$\{x_{\mu_1 + \cdots + \mu_j - \gamma_j + 1}, \dots, x_{\mu_1 + \cdots + \mu_j} \}$.

\begin{lemma}
    \label{lem:parabolic-spanning}
    For any partition $\mu \vdash n$, the collection of superspace elements
    \begin{equation}
    \label{eq:parabolic-spanning-set}
        \bigsqcup_{\zero \leq \gamma \leq \mu} \varepsilon_\mu \cdot (\AAA_n(\mu,\gamma) \cdot \theta_{J(\mu,\gamma)})
    \end{equation}
    descends to a spanning set of $\varepsilon_\mu \cdot SR_n$.
\end{lemma}

\begin{proof}
    Theorem~\ref{thm:artin-basis} guarantees that $\{ \varepsilon_\mu \cdot m \,:\, m \in \AAA_n \}$ descends to a  spanning set of $\varepsilon_\mu \cdot SR_n$. The set \eqref{eq:parabolic-spanning-set} is obtained from $\{ \varepsilon_\mu \cdot m \,:\, m \in \AAA_n \}$ by 
    \begin{enumerate}
        \item removing instances where $\varepsilon_\mu \cdot m = 0$, and
        \item when $\varepsilon_\mu \cdot m = \pm \varepsilon_\mu \cdot m' \neq 0$  for $m,m' \in \AAA_n$, keeping only one of $\varepsilon_\mu \cdot m, \varepsilon_\mu \cdot m'$.
    \end{enumerate}
    The lemma follows.
\end{proof}

The main goal of this section is to show that the spanning set of Lemma~\ref{lem:parabolic-spanning} is actually a basis. Since the Fields Conjecture 2 predicts an isomorphism $SR_n \cong \FF[\OP_n] \otimes \sign$ of ungraded $\symm_n$-modules, the next lemma shows that this is a worthy endeavor.

\begin{lemma}
    \label{lem:A-combinatorial-interpretation}
    For any partition $\mu \vdash n$ we have 
    \begin{equation}
    \label{eq:A-and-OP}
        \sum_{\zero \leq \gamma \leq \mu} \# \AAA_n(\mu,\gamma) = \dim_\FF \varepsilon_\mu \cdot (\FF[\OP_n ] \otimes \sign).
    \end{equation}
\end{lemma}

\begin{proof}
    Write $\mu = (\mu_1,\dots,\mu_p)$ and let $\OP_n(\mu)$ be the family of ordered set partitions $\sigma$ of $[n]$ such that the entries $\mu_1 + \cdots + \mu_{j-1} + 1, \mu_1 + \cdots + \mu_{j-1} + 2, \dots, \mu_1 + \cdots + \mu_{j-1} + \mu_j$ appear in increasing order from left to right for $j = 1,\dots,p$. We refer to this set of entries as the $j^{th}$ {\em batch} of numbers in $[n]$.
    It is clear that the right-hand side of \eqref{eq:A-and-OP} equals $\# \OP_n(\mu)$.

    Starting with the empty ordered set partition with no blocks, an element $\sigma \in \OP_n(\mu)$ may be built up batch by batch as follows. In the first stage, we form a sequence of $\mu_1 - \gamma_1$ singleton sets $(\bullet \mid \cdots \mid \bullet)$ with the entry $\bullet$ for some $0 \leq \gamma_1 \leq \mu_1$ . Then, we add $\gamma_1$ $\circ$'s to these $\mu_1 - \gamma_1$ sets in some fashion (possibly with multiplicity). The number of ways to do this is ${\mu_1 - 1 \choose \gamma_1}$ (note that this quantity is 0 when $\mu_1 = \gamma_1$). Finally, we replace all of the $\bullet$'s and $\circ$'s from left to right with the first batch of numbers $1,2,\dots,\mu_1$. This gives an ordered set partition in $\OP_{(\mu_1)}(\mu_1)$.

    Now suppose $j > 1$ and we are given an ordered set partition \[\sigma = (B_1 \mid \cdots \mid B_s) \in \OP_{\mu_1+\cdots+\mu_{j-1}}(\mu_1, \dots, \mu_{j-1})\]
    with $s$ blocks.
    The stages $1,2,\dots,j-1$ involved in the construction of $\sigma$ have inductively resulted in numbers $\gamma_1,\gamma_2,\dots,\gamma_{j-1}$ which satisfy  $0 \leq \gamma_r \leq \mu_r$ and 
    \[\gamma_1 + \gamma_2 + \cdots + \gamma_{j-1} = \mu_1 + \cdots + \mu_{j-1} - p.\]
    Here $\mu_r - \gamma_r$ is the number of blocks added at stage $r$ for $r = 1,\dots,j-1$.
    We start by adding $\mu_j - \gamma_j$ new singleton blocks with the lone entry $\bullet$ to $\sigma$ for some $0 \leq \gamma_j \leq \mu_j$. Since $\sigma$ has $s$ blocks, the number of ways to do this is 
    \[ {s + \mu_j - \gamma_j \choose \mu_j - \gamma_j} = {(\mu_1 - \gamma_1) + \cdots + (\mu_j - \gamma_j) \choose \mu_j - \gamma_j}. \]
    This yields a figure with $s + \mu_j - \gamma_j$ blocks. Next, we add $\gamma_j$ copies of $\circ$ to these blocks in some fashion (possibly with multiplicity). The number of ways to do this is 
    \[ {s + \mu_j - 1 \choose \gamma_j} =
    {(\mu_1 - \gamma_1) + \cdots + (\mu_{j-1} - \gamma_{j-1}) + \mu_j - 1 \choose \gamma_j}.\]
    Finally, we replace all of the $\bullet$'s and $\circ$'s from left to right with the $j^{th}$ batch of numbers from $[n]$. This gives an ordered set partition in $\OP_{\mu_1 + \cdots 
 + \mu_j}(\mu_1, \dots, \mu_j)$.

    The above enumeration shows that the cardinality of $\OP_n(\mu)$ is 
    \begin{multline}
        \# \OP_n(\mu) = \\ \sum_{\zero \leq \gamma\leq \mu} \left[ \prod_{j=1}^p {(\mu_1 - \gamma_1) + \cdots + (\mu_j - \gamma_j) \choose \mu_j - \gamma_j} \cdot  {(\mu_1 - \gamma_1) + \cdots + (\mu_{j-1} - \gamma_{j-1}) + \mu_j - 1 \choose \gamma_j} \right].
    \end{multline}
   It is not hard to see 
    \begin{equation}
        \# \AAA_n(\mu,\gamma) = \prod_{j=1}^p {(\mu_1 - \gamma_1) + \cdots + (\mu_j - \gamma_j) \choose \mu_j - \gamma_j} \cdot  {(\mu_1 - \gamma_1) + \cdots + (\mu_{j-1} - \gamma_{j-1}) + \mu_j - 1 \choose \gamma_j}
    \end{equation}
    from which the result follows.
\end{proof}

The recursive proof of Lemma~\ref{lem:A-combinatorial-interpretation} is best understood by example. Suppose  $\mu = (5,3,3,3,2) \vdash 16$ and $\gamma = (2,1,3,0,1)$. The construction of one possible ordered set partition in $\OP_{16}$ is shown in Figure~\ref{fig:osp-construction}.
\begin{figure}
\begin{footnotesize}
\begin{align*}
    &\varnothing \\
    &( \bullet \mid \bullet \mid \bullet) \\
    & ( \bullet, \, \circ  \mid \bullet \mid \bullet, \, \circ ) \\
    & (1, \, 2 \mid 3 \mid 4, \, 5) \\
    & (1, \, 2 \mid \bullet \mid  3 \mid 4, \, 5 \mid \bullet ) \\
    & (1, \, 2, \, \circ \mid \bullet \mid  3 \mid 4, \, 5 \mid \bullet ) \\
    & (1, \, 2, \, 6 \mid 7 \mid  3 \mid 4, \, 5 \mid 8 ) \\
    & (1, \, 2, \, 6 \mid 7 \mid  3 \mid 4, \, 5 \mid 8 ) \\
    & (1, \, 2, \, 6 \mid 7 \mid  3, \, \circ, \, \circ \mid 4, \, 5, \, \circ \mid 8 ) \\
    & (1, \, 2, \, 6 \mid 7 \mid  3, \, 9, \, 10 \mid 4, \, 5, \, 11 \mid 8 ) \\
    & (\bullet \mid 1, \, 2, \, 6 \mid 7 \mid  3, \, 9, \, 10 \mid \bullet \mid \bullet \mid  4, \, 5, \, 11 \mid 8 ) \\
    & (\bullet \mid 1, \, 2, \, 6 \mid 7 \mid  3, \, 9, \, 10 \mid \bullet \mid \bullet \mid  4, \, 5, \, 11 \mid 8 ) \\
    & (12 \mid 1, \, 2, \, 6 \mid 7 \mid  3, \, 9, \, 10 \mid 13 \mid 14 \mid  4, \, 5, \, 11 \mid 8 ) \\
    & (12 \mid 1, \, 2, \, 6 \mid 7 \mid  3, \, 9, \, 10 \mid 13 \mid 14 \mid  4, \, 5, \, 11 \mid \bullet  \mid 8 ) \\
    & (12 \mid 1, \, 2, \, 6 \mid 7 \mid  3, \, 9, \, 10 \mid 13 \mid 14 \mid  4, \, 5, \, 11 \mid \bullet, \, \circ  \mid 8 ) \\
    & (12 \mid 1, \, 2, \, 6 \mid 7 \mid  3, \, 9, \, 10 \mid 13 \mid 14 \mid  4, \, 5, \, 11 \mid 15, \, 16  \mid 8 ) 
\end{align*}
\end{footnotesize}
\caption{Constructing an ordered set partition for $\mu = (5,3,3,3,2) \vdash 16$ and $\gamma = (2,1,3,0,1)$.}
\label{fig:osp-construction}
\end{figure}


\subsection{Parabolic $\DDD$-operators}  Lemma~\ref{lem:A-combinatorial-interpretation} shows that antisymmetrizing the basis $\AAA_n$ of $SR_n$ naturally leads to a potentially useful spanning set of $\varepsilon_\mu \cdot SR_n$. In order to show that this spanning set is linearly independent, we need to bound the dimension of $\varepsilon_\mu \cdot SR_n$ from below. To do this, we switch to the inverse system and consider $\varepsilon_\mu \cdot SH_n$.
The following result gives a  source of $\mu$-antisymmetric elements of $SH_n$. We write
\begin{equation}
    \FF[\xx_n]^{\symm_\mu} := \{ f \in \FF[\xx_n] \,:\, w \cdot f = f \text{ for all $w \in \symm_\mu$} \}
\end{equation}
for the $\symm_\mu$-invariant polynomials in $\FF[\xx_n]^{\symm_\mu}$.

\begin{lemma}
    \label{lem:alternating-closure}
    Let $\mu \vdash n$ be a partition of $n$ and let $\FF[\xx_n]^{\symm_\mu} \subset \FF[\xx_n]$ be the subalgebra of $\symm_\mu$-invariant polynomials. The space $\varepsilon_\mu \cdot SH_n$ satisfies the following properties.
    \begin{enumerate}
        \item We have $\delta_n \in \varepsilon_\mu \cdot SH_n$.
        \item The space $\varepsilon_\mu \cdot SH_n$ is closed under the operators $d_1,d_2,\dots,d_{n-1}$, and
        \item The space $\varepsilon_\mu \cdot SH_n$ is closed under the operator $f \odot (-)$ for any $f \in \FF[\xx_n]^{\symm_\mu}$.
    \end{enumerate}
\end{lemma}

\begin{proof}
    (1) is clear. 
    Theorem~\ref{thm:operator} implies that the operators in (2) and (3) preserve $SH_n$. We leave it for the reader to check that
    \begin{equation}
        \varepsilon_\mu \cdot d_j(g) = d_j(\varepsilon_\mu \cdot g) \quad \text{and} \quad f \odot (\varepsilon_\mu \cdot g) = \varepsilon_\mu \cdot (f \odot g)
    \end{equation}
    for all $1 \leq j \leq n-1$, $g \in \Omega_n$, and $f \in \FF[\xx_n]^{\symm_\mu}$. The result follows.
\end{proof}

We use Lemma~\ref{lem:alternating-closure} to build superspace operators $\DDD^\TT_\mu: \Omega_n \to \Omega_n$ for which $\DDD^\TT_\mu(\delta_n) \in \varepsilon_\mu \cdot SH_n$. These operators depend on $\mu$ and another piece of combinatorial data $\TT$ defined as follows.

\begin{definition}
    \label{def:translation-sequences}
    Let $\mu = (\mu_1,\dots,\mu_p) \vdash n$. A {\em $\mu$-translation sequence} is a list $\TT = (T_1,\dots,T_p)$ of sets where 
    \[ T_j \subseteq \{ \mu_1 + \cdots + \mu_{j-1} + 1,\dots,\mu_1 + \cdots + \mu_{j-1} + \mu_j \} \text{ for all $1 \leq j \leq p$.}\]
    We write $\gamma(\TT) := (\# T_1,\dots,\# T_p)$ for the sequence of cardinalities of these sets.
\end{definition}

For example, suppose $\mu = (5,4,4,3) \vdash 16$ so that $p=4$. The partition of $[16]$ induced by $\mu$ is 
\[ ( 1, \, 2, \, 3, \, 4, \, 5 \mid 6, \, 7, \, 8, \, 9 \mid 10, \, 11, \, 12, \, 13 \mid 14, \, 15, \, 16). \]
One possible $\mu$-translation sequence is
\[ \TT = (T_1,T_2,T_3,T_4) \quad \text{where} \quad T_1 = \{1,3,4\}, \, T_2 = \{6,7,8,9\}, \, T_3 = \{ 11, \, 13 \}, \text{ and } T_4 = \varnothing.\]
We have $\gamma(\TT) = (3,4,2,0)$ so that $(\mu,\gamma(\TT))$ is a signed partition. We attach polynomial weights to translation sequences as follows.

\begin{definition}
    \label{def:translation-weight}
    Let $\mu = (\mu_1,\dots,\mu_p) \vdash n$ and let $\TT = (T_1,\dots,T_p)$ be a $\mu$-translation sequence with $\gamma(\TT) = (\gamma_1,\dots,\gamma_p)$. For $1 \leq j \leq p$, write
    $T_j = \{ t^{(j)}_1 < \cdots < t^{(j)}_{\gamma_j} \}.$ The {\em weight} $\sss(\TT) \in \FF[\xx_n]$ of $\TT$ is the product of Schur polynomials 
    \begin{equation}
        \sss(\TT) := \prod_{j=1}^p s_{\nu(T_j)} (x_{\mu_1 + \cdots + \mu_j - \gamma_j + 1},\dots,x_{\mu_1 + \cdots + \mu_j})
    \end{equation}
    where $\nu(T_j)$ is the partition of length $\leq \gamma_j$ given by
    \begin{equation}
        \nu(T_j) := ( \mu_1 + \cdots + \mu_j - \gamma_j + 1 - t^{(j)}_1 \geq \cdots \geq \mu_1 + \cdots + \mu_j - t^{(j)}_{\gamma_j}).
    \end{equation}

\end{definition}

In our example from before Definition~\ref{def:translation-weight}, the polynomial $\sss(\TT) \in \FF[\xx_{16}]$ for $\TT = (T_1,T_2,T_3,T_4)$ is given by
\[ \sss(\TT) = s_{211}(x_3,x_4,x_5) \times s_{\varnothing}(x_6,x_7,x_8,x_9) \times s_1(x_{12},x_{13}) \times 1.\]
In general, the degree of $\sss(\TT)$ increases as the sets $T_j \subseteq \{ \mu_1 + \cdots + \mu_{j-1} + 1,\dots,\mu_1 + \cdots + \mu_{j-1} + \mu_j\}$ `translate away from' Gale-maximality, whence the notation.

The operators $\DDD^\TT_\mu: \Omega_n \to \Omega_n$ will be defined implicitly. To set the stage, introduce a new set of commuting variables $\yy = (y_1,y_2,\dots)$. Given $\mu \vdash n$ and $0 \leq r \leq n$, the {\em power matrix} $P_r(\yy,\mu)$ is the $r \times n$ matrix whose $(i,j)$-entry is $y_i^{n-j+1}$. When $n = 8, \mu = (3,3,2),$ and $r = 3$ the matrix $P_r(\yy,\mu)$ is as follows:
\begin{footnotesize}
\[ 
P_r(\yy,\mu) = \left( \begin{array}{ccc|ccc|cc}
    y_1^8 & y_1^7 & y_1^6 & y_1^5 & y_1^4 & y_1^3 & y_1^2 & y_1 \\
    y_2^8 & y_2^7 & y_2^6 & y_2^5 & y_2^4 & y_2^3 & y_2^2 & y_2 \\
    y_3^8 & y_3^7 & y_3^6 & y_3^5 & y_3^4 & y_3^3 & y_3^2 & y_3 \\
    \end{array}
\right).\]
\end{footnotesize}

\noindent
We place vertical bars after columns $\mu_1, \mu_1 + \mu_2$, and so on. The entries of $P_r(\yy,\mu)$ do not depend on $\mu$ and do not involve the variables $\xx_n = (x_1,\dots,x_n)$. 
Our second matrix  differs on both counts. For $\mu \vdash n$ and $0 \leq r \leq n$, the {\em factor matrix} $F_r(\yy,\mu)$ has entries
\begin{equation}
    F_r(\yy,\mu)_{i,j} := y_i^{j - \mu_1 - \cdots - \mu_{k-1}} \prod_{m = \mu_1 + \cdots + \mu_k+1}^n (y_i - x_m)
\end{equation}
where $k$ is minimal so that $j-1 \leq \mu_1  +\cdots + \mu_k$. In our example we have $F_r(\yy,\mu) = ( A \mid B \mid C)$ where 
\begin{footnotesize}
\[ 
A = \left( \begin{array}{ccc}
    y_1^3(y_1-x_4)(y_1-x_5) \cdots (y_1-x_8) & y_1^2(y_1-x_4)(y_1-
    x_5) \cdots (y_1-x_8) &
    y_1(y_1-x_4)(y_1-x_5) \cdots (y_1-x_8) \\
    y_2^3(y_2-x_4)(y_2-x_5) \cdots (y_2-x_8) & y_2^2(y_2-x_4)(y_2-x_5) \cdots (y_1-x_8) &
    y_2(y_2-x_4)(y_2-x_5) \cdots (y_2-x_8) \\
    y_3^3(y_3-x_4)(y_3-x_5) \cdots (y_3-x_8) & y_3^2(y_3-x_4)(y_3-x_5) \cdots (y_3-x_8) &
    y_3(y_3-x_4)(y_3-x_5) \cdots (y_3-x_8) \\
    \end{array}
\right),\]
\[
B = \begin{pmatrix}
    y_1^3(y_1-x_7)(y_1-x_8) & y_1^2(y_1 - x_7)(y_1-x_8) & y_1(y_1-x_7)(y_1-x_8) \\
    y_2^3(y_2-x_7)(y_2-x_8) & y_2^2(y_2 - x_7)(y_2-x_8) & y_2(y_2-x_7)(y_2-x_8) \\
    y_3^3(y_3-x_7)(y_3-x_8) & y_3^2(y_3 - x_7)(y_3-x_8) & y_3(y_3-x_7)(y_3-x_8)
\end{pmatrix}, \text{ and }
C = \begin{pmatrix}
    y_1^2 & y_1 \\
    y_2^2 & y_2 \\
    y_3^2 & y_3 
\end{pmatrix}.
\]
\end{footnotesize}
The following simple linear algebra result will unlock the operators $\DDD^\TT_\mu$.

\begin{observation}
    \label{obs:column-operations}
    There exists a lower triangular $n \times n$ matrix $C(\mu)$ with $1$'s on the diagonal and entries in $\FF[\xx_n]^{\symm_\mu}$ such that
    \[ F_r(\yy,\mu) = P_r(\yy,\mu) \cdot C(\mu).\]
\end{observation}

    Observation~\ref{obs:column-operations} says that  $P_r(\yy,\mu)$ can be transformed into $F_r(\yy,\mu)$ by a sequence of column operations which add $\FF[\xx_n]^{\symm_\mu}$-multiples of column $j_1$ to column $j_2$ for $j_2 < j_1$. This may be done by processing the columns of $P_r(\yy,\mu)$ from left to right, adding or subtracting multiples of columns lying in subsequent blocks by appropriate elementary symmetric polynomials in `terminal' variable sets of the form $\{ x_{\mu_1 + \cdots + \mu_j + 1},\dots,x_n\}$. In the $\mu = (3,3,2)$ example above, this is done as follows.
    \begin{enumerate}
    \item Subtract $e_1(x_4,\dots,x_8)$ times column 2 from column 1, add $e_2(x_4,\dots,x_8)$ times column 3 to column 1,  subtract $e_3(x_4,\dots,x_8)$ times column 4 from column 1,
     add $e_4(x_4,\dots,x_8)$ times column 5 to column 1, and finally subtract $e_5(x_4,\dots,x_8)$ times column 6 from column 1.
    \item Subtract $e_1(x_4,\dots,x_8)$ times column 3 from column 2, add $e_2(x_4,\dots,x_8)$ times column 4 to column 2, $\dots$ and  subtract $e_5(x_4,\dots,x_8)$ times column 7 from column 2.
    \item Subtract $e_1(x_4,\dots,x_8)$ times column 4 from column 3, add $e_2(x_4,\dots,x_8)$ times column 5 to column 3, $\dots$ and  subtract $e_5(x_4,\dots,x_8)$ times column 8 from column 3.
    \item Subtract $e_1(x_7,x_8)$ times column 5 from column 4, then add $e_2(x_7,x_8)$ times column 6 to column 4.
    \item Subtract $e_1(x_7,x_8)$ times column 6 from column 5, then add $e_2(x_7,x_8)$ times column 7 to column 5.
    \item Subtract $e_1(x_7,x_8)$ times column 7 from column 6, then add $e_2(x_7,x_8)$ times column 8 to column 6.
\end{enumerate}

Let $\TT = (T_1,\dots,T_p)$ be a $\mu$-translation sequence for $\mu  \vdash n$ and consider the disjoint union $T := T_1 \sqcup \cdots \sqcup T_p$. The {\em augmented factor matrix} is the $n \times n$ matrix with block form 
\begin{equation}
    \widetilde{F}(\yy,\mu,\TT) := \begin{pmatrix} F_{\# T}(\yy,\mu) \\ \hline E\end{pmatrix}
\end{equation}
where $E$ is the $(n-\#T) \times n$ row echelon 0,1-matrix with 1's in the columns which are not contained in $T$. If $\mu = (3,3,2) \vdash 8$ and $\TT = ( 2 \mid  4, \, 6 \mid  \varnothing)$, we have $T = \{2,4,6\}$ so that 
\begin{footnotesize}
\[ 
E = \left(\begin{array}{c c c | c c c | c c}
    1 & 0 & 0 & 0 & 0 & 0 & 0 & 0  \\
    0 & 0 & 1 & 0 & 0 & 0 & 0 & 0 \\
    0 & 0 & 0 & 0 & 1 & 0 & 0 & 0 \\
    0 & 0 & 0 & 0 & 0 & 0 & 1 & 0 \\
    0 & 0 & 0 & 0 & 0 & 0 & 0 & 1
\end{array}\right).\]
\end{footnotesize}
 
 \noindent
 For any  $J \subseteq [n]$ with $\#J = \# T$, let $\widetilde{F}(\xx_J,\mu,\TT)$ be the matrix obtained by replacing the $\yy$-variables in $\widetilde{F}(\yy,\mu,\TT)$ by $\{x_j \,:\, j \in J\}$ in some order. We only care about the determinant of $\widetilde{F}(\xx_J,\mu,\TT)$ up to sign, so this order is immaterial. The following result shows that $\det \widetilde{F}(\xx_J,\mu,\TT)$ has favorable vanishing properties. Recall the Gale order $\leq_\Gale$ on subsets of $[n]$ of the same size and the polynomial $f_J \in \FF[\xx_n]$ for $J \subseteq [n]$. Here and hereafter $\doteq$ will denote equality up to sign.

\begin{lemma}
    \label{lem:augmented-determinant} 
    Let $\mu = (\mu_1,\dots,\mu_p) \vdash n$ and let $\TT = (T_1,\dots,T_p)$ be a $\mu$-translation sequence. Let $\gamma := \gamma(\TT)$ and write $T := T_1 \sqcup \cdots \sqcup T_p$.
    For a subset $J \subseteq [n]$ of size $\# T$, let $\PPP_{T,J} \in \FF[\xx_n]$ be defined up to sign by
    \begin{equation}
        \PPP_{T,J} := \det \widetilde{F}(\xx_J,\mu,\TT).
    \end{equation}  
    \begin{enumerate}
        \item If $J \not\leq_\Gale J(\mu,\gamma)$ we have $\PPP_{T,J} = 0$.
        \item The polynomial $\PPP_{T,J(\mu,\gamma)}$ is given by
        \begin{equation}
            \PPP_{T,J(\mu,\gamma)} \doteq f_{J(\mu,\gamma)} \times \sss(\TT).
        \end{equation}
    \end{enumerate}
\end{lemma}

\begin{proof}
    This proof makes use of the bialternant formula \eqref{eq:bialternant-formula} for the Schur polynomial.
     We have  
    \[ \det \widetilde{F}(\xx_J,\mu,\TT) \doteq \det \overline{F}(\xx_J,\mu,\TT)\]
    where $\overline{F}(\xx_J,\mu,\TT)$ is the $\# T \times \# T$ matrix formed by the columns in the upper block of $\widetilde{F}(\xx_J,\mu,\TT)$ indexed by $T$.
    If $J \not\leq_\Gale J(\mu,\gamma)$, after rearranging rows  the matrix $\overline{F}(\xx_J,\mu,\TT)$ has the form 
    \[ \begin{pmatrix} {\bf *} & {\bf *} \\ {\bf 0} & {\bf *}
    \end{pmatrix}\]
    where the southwest block of zeros intersects the main diagonal. Any such matrix has zero determinant.

    On the other hand (after possibly rearranging rows), the determinant $\PPP_{T,J(\mu,\gamma)}$ is block upper triangular  with diagonal block sizes $\gamma(\TT) = (\gamma_1,\dots,\gamma_p)$. We may factor out an $(x_j - x_i)$ from the columns of this determinant for each $j < i$ such that $j \in T$ and $i$ lies in a block after that of $j$; these are among factors of the polynomial $f_{J(\mu,\gamma)}$. We may also factor out an $x_j$ for each $j \in J(\mu,\gamma)$ from the rows of this determinant; these are other factors of $f_{J(\mu,\gamma)}$. The diagonal block determinants are alternants. Their Vandermonde factors account for the remaining factors of $f_{J(\mu,\gamma)}$.\footnote{These have the form $(x_j -x_i)$ where $j$ lies in $T$, and $i > j$ lies in the same block as $j$.} The bialternant formula shows that what remains is the product $\sss(\TT)$ of Schur polynomials corresponding to $\TT$.
\end{proof}

We continue our running example of $\mu = (3,3,2)$ and $\TT = (T_1,T_2,T_3)$ with $T_1 = \{2\}, T_2 = \{4,6\},$ and $T_3 = \varnothing$. Then $\gamma = (\gamma_1,\gamma_2,\gamma_3) = (1,2,0)$ and $J(\mu,\gamma) = \{3,5,6\}$. We have
\begin{footnotesize}
\[ \PPP_{T,J} \doteq
\begin{vmatrix}
    x_{j_1}^2(x_{j_1} - x_4) \cdots (x_{j_1} - x_8) & x_{j_1}^3(x_{j_1}-x_7)(x_{j_1} - x_8) & x_{j_1}(x_{j_1}-x_7)x_{j_1} - x_8) \\
    x_{j_2}^2(x_{j_2} - x_4) \cdots (x_{j_2} - x_8) & x_{j_2}^3(x_{j_2}-x_7)(x_{j_2} - x_8) & x_{j_2}(x_{j_2}-x_7)(x_{j_2} - x_8) \\
    x_{j_3}^2(x_{j_3} - x_4) \cdots (x_{j_3} - x_8) & x_{j_3}^3(x_{j_3}-x_7)(x_{j_3} - x_8) & x_{j_3}(x_{j_3}-x_7)(x_{j_3} - x_8) 
\end{vmatrix}
\]
\end{footnotesize}

\noindent
for $J = \{ j_1 < j_2 < j_3 \} \subseteq [8]$. This determinant vanishes for support reasons unless $J \leq_\Gale \{3,5,6\}$. Furthermore, we have 
\begin{footnotesize}
\begin{align*} \PPP_{T,356} &\doteq
\begin{vmatrix}
    x_{3}^2(x_{3} - x_4) \cdots (x_{3} - x_8) & x_{3}^3(x_{3}-x_7)(x_{3} - x_8) & x_{3}(x_{3}-x_7)(x_{3} - x_8) \\
    0 & x_{5}^3(x_{5}-x_7)(x_{5} - x_8) & x_{5}(x_{5}-x_7)(x_{5} - x_8) \\
    0 & x_{6}^3(x_{6}-x_7)(x_{6} - x_8) & x_{6}(x_{6}-x_7)(x_{6} - x_8) 
\end{vmatrix} \\
&=  \begin{vmatrix} 
 x_{3}^2(x_{3} - x_4) \cdots (x_{3} - x_8)
\end{vmatrix} \cdot \begin{vmatrix}
    x_{5}^3(x_{5}-x_7)(x_{5} - x_8) & x_{5}(x_{5}-x_7)(x_{5} - x_8) \\
    x_{6}^3(x_{6}-x_7)(x_{6} - x_8) & x_{6}(x_{6}-x_7)(x_{6} - x_8) 
\end{vmatrix} \\
&=  \frac{f_{356}}{(x_5-x_6)} \times \begin{vmatrix} x_3 \end{vmatrix} \cdot \begin{vmatrix} x_5^2 & 1 \\ x_6^2 & 1 \end{vmatrix} \\
&=  f_{356} \times s_1(x_3) \cdot s_1(x_5,x_6) \\
&= f_{J(\mu,\gamma)} \times \sss(\TT).
\end{align*}
\end{footnotesize}

Given a partition $\mu \vdash n$ and a $\mu$-translation sequence $\TT = (T_1,\dots,T_p)$ with $\#(T_1 \sqcup \cdots \sqcup T_p) = r$, let $C(\mu)$ be a $n \times n$ matrix over $\FF[\xx]^{\symm_\mu}$ as in Observation~\ref{obs:column-operations}. Since $C(\mu)$ is lower triangular with 1's on the diagonal, it has an inverse $C(\mu)^{-1}$ which is also defined over $\FF[\xx_n]^{\symm_n}$ and lower unitriangular. We have the matrix equation
 \begin{equation}
        \widetilde{F}(\yy,\mu,\TT) \cdot C(\mu)^{-1} = \begin{pmatrix}
            F_r(\yy,\mu) \\ \hline E
        \end{pmatrix} \cdot C(\mu)^{-1} = 
        \begin{pmatrix}
            P_r(\yy,\mu) \\ \hline H
        \end{pmatrix}
\end{equation}
where $H := E \cdot C(\mu)^{-1}$ is an $(n-r) \times n$ matrix defined over $\FF[\xx_n]^{\symm_n}$ which depends on $\TT$. We use $H$ to define our operators $\DDD^\TT_\mu$. Given a subset $K \subseteq [n]$, let $K^* := \{n-k+1 \,:\, k \in K\}$. If $K = \{k_1 < \cdots < k_r \}$ we write 
\begin{equation}
    d_K := d_{k_1} \cdots d_{k_r}: \Omega_n \longrightarrow \Omega_n
\end{equation}
for the corresponding composition of higher Euler operators.

\begin{defn}
    \label{def:D-operators}
    Let $\mu \vdash n$, let $\TT = (T_1,\dots,T_p)$ be a $\mu$-translation sequence with $\# T_1 + \cdots + \# T_p = r$, and let $H$ be as above. We define an operator $\DDD_\mu^\TT$ on $\Omega_n$ by the formula
     \begin{equation}
        \DDD_\mu^\TT: \Omega_n \longrightarrow \Omega_n, \quad \quad \DDD^\TT_\mu(f) := \sum_{\# I = n-r} (-1)^{\sum I} \Delta_I(H) \odot d_{([n]-I)^*}(f).
    \end{equation}
    Here $\sum I := \sum_{i \in I} i$ and $\Delta_I(H)$ denotes the maximal minor of $H$ with columns indexed by $I$.
\end{defn}

The operators $\DDD^\TT_\mu$ depend on the matrix $H$, and therefore {\em a priori} on the matrix $C(\mu)$ of Observation~\ref{obs:column-operations}. It is not hard to see that $C(\mu)$ is uniquely determined; this (non)dependence will play no role and we leave it implicit.

\begin{lemma}
    \label{lem:abstract-leading}
    Let $\mu = (\mu_1,\dots,\mu_p) \vdash n$, let $\TT = (T_1,\dots,T_p)$ be a $\mu$-translation sequence, and let 
    $\gamma := \gamma(\TT)$.
    If $1 \notin T_1$, then $\DDD^\TT_\mu(\delta_n) \in \varepsilon_\mu \cdot SH_n$ is nonzero and has unique Gale-maximal fermionic monomial $\theta_{J(\mu,\gamma)}$. The coefficient of $\theta_{J(\mu,\gamma)}$ in $\DDD^\TT_\mu(\delta_n)$ is 
    \[ \pm (\sss(\TT) \times f_{J(\mu,\gamma)}) \odot \delta_n.\]
\end{lemma}

\begin{proof}
    Let $T = T_1 \sqcup \cdots \sqcup T_p$ and let $r := \# T$. 
   Since the matrix $H$ in Definition~\ref{def:D-operators} is defined over $\FF[\xx_n]^{\symm_n}$, Lemma~\ref{lem:alternating-closure} furnishes the membership 
    \begin{equation}
    \DDD^\TT_\mu(\delta_n) \in \varepsilon_\mu \cdot SH_n. 
    \end{equation}
    We have the chain of up-to-sign equalities:
    \begin{align}
        \text{coefficient of $\theta_J$ in } \DDD_\mu^\TT (\delta_n) &= \text{coefficient of $\theta_J$ in  } \sum_{\# I = n-r} (-1)^{\sum I} \Delta_I(H) \odot d_{([n]-I)^*}(\delta_n) \\
        &\doteq \left[ \det \left( \begin{array}{c}  P_r(\xx_J,\mu) \\ \hline H \end{array} \right) \right] \odot \delta_n \\
        &= \left[ \det \left( \begin{array}{c}  F_r(\xx_J,\mu) \\ \hline E \end{array} \right) \right] \odot \delta_n \\
        &= \det \widetilde{F}(\xx_J,\mu,\TT) \odot \delta_n \\
        &\doteq \PPP_{T,J} \odot \delta_n
    \end{align}
where $\PPP_{T,J}$ is as in Lemma~\ref{lem:augmented-determinant}.
The first $=$ is the definition of $\DDD^\TT_\mu$, the next $\doteq$ comes from the definition of the higher Euler operators, the next $=$ follows because the matrix $C(\mu)$ in Observation~\ref{obs:column-operations} is lower unitriangular, the next $=$ is the definition of $\widetilde{F}(\xx_J,\mu,\TT)$, and the last $\doteq$ is the definition of $\PPP_{T,J}$. Lemma~\ref{lem:augmented-determinant} implies that the coefficient of $\theta_J$ in $\DDD^\TT_\mu(\delta_n)$ vanishes unless $J \leq_\Gale J(\mu,\gamma)$. Furthermore, if $J = J(\mu,\gamma)$, Lemma~\ref{lem:augmented-determinant} yields
\begin{equation}
\label{eq:target-coefficient}
    \text{coefficient of $\theta_{J(\mu,\gamma)}$ in } \DDD_\mu^\TT(\delta_n) \doteq (\sss(\TT) \times f_{J(\mu,\gamma)}) \odot \delta_n \in \FF[\xx_n].
\end{equation}

It remains to show that the polynomial \eqref{eq:target-coefficient} is nonzero. This requires the hypothesis $1 \notin T_1$. Steinberg's Theorem~\ref{thm:steinberg} implies 
\begin{align*}
    (\sss(\TT) \times f_{J(\mu,\gamma)}) \odot \delta_n \neq 0 \quad &\Leftrightarrow \quad \sss(\TT) \times f_{J(\mu,\gamma)} \notin I_n \\
    &\Leftrightarrow \quad \sss(\TT) \notin (I_n : f_{J(\mu,\gamma)}).
\end{align*}
Since $1 \notin T_1$ we have $f_{J(\mu,\gamma)} \notin I_n$ by \cite[Lem. 4.7]{RW}. Writing 
\begin{equation}
    \sss(\TT) = \prod_{j=1}^p s_{\nu(T_j)}(x_{\mu_1 + \cdots + \mu_j - \gamma_j + 1}, \dots, x_{\mu_1 + \cdots + \mu_j})
\end{equation}
as in Definition~\ref{def:translation-weight}, the condition $1 \notin T_1$ implies $\nu(T_1)_1 < \mu_1 - \gamma_1$. One has
\[
\stair(J(\mu,\gamma))= (1,2,\dots,\mu_1-\gamma_1,\dots,\mu_1-\gamma_1 \mid \mu_1-\gamma_1+1,\dots,\mu_1-\gamma_1+\mu_2-\gamma_2,\dots,\mu_1-\gamma_1+\mu_2-\gamma_2\mid \cdots).
\]
It follows that every monomial in the expansion of $\sss(\TT)$ lies in the set $\AAA_n(J(\mu,\gamma))$, and Theorem~\ref{thm:colon-ideal-basis} implies $\sss(\TT) \notin (I_n : f_{J(\mu,\gamma)}).$
\end{proof}

If $1 \in T_1$, the proof of Lemma~\ref{lem:abstract-leading} breaks down in one of two ways.
\begin{itemize}
    \item If $T_1 = \{1,2,\dots,\mu_1\}$ then $1 \in J(\mu,\gamma)$ and \cite[Lem. 4.7]{RW} implies $f_{J(\mu,\gamma)} \in I_n$ so that $(I_n : f_{J(\mu,\gamma)}) = \FF[\xx_n]$ is the unit ideal.
    \item If $T_1$ is a proper subset of $\{1,2,\dots,\mu_1\}$ of size $\# T_1 = \gamma_1$ then one has $\nu(T_1)_1 = \mu_1 - \gamma_1$. The power $x_{\mu_1}^{\mu_1 - \gamma_1}$ appears in the monomial expansion of $\sss(\TT)$, so the monomial expansion of $\sss(\TT)$ contains terms which do not lie in $\AAA_n(J(\mu,\gamma))$.
\end{itemize}

In their proof of the Fields Conjecture (1), Rhoades--Wilson introduced certain explicit $\DDD$-operators \cite[Def. 4.1]{RW}. Our operators $\DDD^\TT_\mu$ are more abstract and, when $\mu = (1^n)$, lead to a simpler proof of the Operator Theorem~\ref{thm:operator}; see Remark~\ref{rmk:alternative-operator}.


\subsection{The set of polynomials $\LLL(m,k,t)$}


Let $(\mu,\gamma)=((\mu_1,\dots,\mu_p),(\gamma_1,\dots,\gamma_p))$ be a signed partition of $n$.
The definition of $\sss(\TT)$ says that, for any $\mu$-translation sequence $\TT$ with $\gamma(\TT)=\gamma$ and $1 \not \in T_1$,
\[
\sss(\TT) = \prod_{j=1}^p s_{\nu_j} (x_{\mu_1+ \cdots+\mu_{j}-\gamma_j+1},\dots,x_{\mu_1+\cdots+\mu_j})
\]
satisfies 
\begin{align}
    \nu_1 \subseteq ( (\mu_1-\gamma_1-1)^{\gamma_1}) \ \ \mbox{ and } \ \ \nu_j \subseteq ( (\mu_j-\gamma_j)^{\gamma_j}) \ \mbox{ for $j=2,\dots,p$,}
\end{align}
and any such $\nu_1,\dots,\nu_p$ can appear.
This fact and Lemma \ref{lem:abstract-leading} give the following lower bound of the dimension of $\varepsilon_\mu \cdot SH_n$.

\begin{lemma}
\label{lem:additionbysatoshi}
Let $\mu$ be a partition of $n$. For every $\bf 0 \leq \gamma \leq \mu$,
let
$R_{\mu,\gamma}$ be the $\FF[\xx_n]^{\symm_\mu}$-submodule of $\FF[\xx_n]/(I_n:f_{J(\mu,\gamma)})$ generated by
\[
Z_{\mu,\gamma}= \left\{ \prod_{j=1}^p s_{\nu_j} (x_{\mu_1+\cdots+\mu_j-\gamma_j+1},\dots,x_{\mu_1+\cdots+\mu_j}) \,:\, 
\begin{array}{l}
\nu_1 \subseteq ( ( \mu_1-\gamma_1-1)^{\gamma_1})\mbox{ and } \\
  \nu_j \subseteq ( ( \mu_j-\gamma_j)^{\gamma_j}) \mbox{ for } j=2,3,\dots,p 
\end{array} \right\}.
\]
Then we have
\[
\dim_\FF \varepsilon_\mu \cdot SH_n \geq \sum_{\bf 0 \leq \gamma \leq \mu} \dim_\FF R_{\mu,\gamma}.
\]
\end{lemma}

\begin{proof}
Since the space $\varepsilon_\mu \cdot SH_n$ is closed under the operation $f \odot (-) $ for any $f \in \FF[\xx_n]^{\symm_\mu}$,
by Lemma \ref{lem:abstract-leading}, for any non-zero polynomial $g \in \FF[\xx_n]^{\symm_\mu} \cdot Z_{\mu,\gamma}$ with $(g\times f_{J(\mu,\gamma)}) \odot \delta_n \ne 0$, there is an element $\eta \in \varepsilon_\mu \cdot SH_n$ that has the unique Gale-maximal fermionic monomial $\theta_{J(\mu,\gamma)}$ with coefficient
\[
(g \times f_{J(\mu,\gamma)}) \odot \delta_n.
\]
Steinberg's Theorem~\ref{thm:steinberg} implies that a set 
 $\{g_1,\dots,g_m\}$ of polynomials is linearly independent in $\FF[\xx_n]/(I_n:f_{J(\mu,\gamma)})$ if and only if \[\{(g_1\times f_{J(\mu,\gamma)})\odot \delta_n,\dots,(g_m\times f_{J(\mu,\gamma)})\odot \delta_n\}\] is linearly independent in $\FF[\xx_n]$.
Consequently, the above construction gives $\dim_\FF R_{\mu,\gamma}$ linearly independent elements of $\varepsilon_\mu \cdot SH_n$ which have the unique Gale-maximal fermionic monomial $\theta_{J(\mu,\gamma)}$.
Since $\theta_{J(\mu,\gamma)}$ differs for different $\gamma$, this proves the desired lower bound.
\end{proof}

In our running example of $\mu=(3,3,2)$ and $\gamma=(1,2,0)$,
the module $R_{\mu,\gamma}$ is generated by the following polynomials
\[
\{s_1(x_3),1\}
\cdot
\{s_{11}(x_5,x_6),s_{1}(x_5,x_6),1\} \cdot \{1\}
\]
as an $\FF[\xx_8]^{\symm_\mu}$-submodule,
where the notation $\{-\}\cdot \{-\}$ means to multiply a polynomial in the first set with a polynomial in the second set in all possible ways.

We will prove  $\dim_\FF R_{\mu,\gamma} = \# \AAA_n(\mu,\gamma)$ in the next subsection. To do this we introduce the following set of polynomials in $\FF[\xx_m]$ for $m \geq 0$.


\begin{definition}
    \label{defn:L-polynomials}
    Let $m,k,t \geq 0$ satisfy $k \leq m$. Define  $\LLL(m,k,t) \subseteq \FF[\xx_m]$ by
\begin{equation}
    \LLL(m,k,t) :=  \left\{
        e_\lambda(x_1,\dots,x_m)\cdot s_\nu(x_{m-k+1},\dots,x_m) \,:\,
        \lambda \subseteq (m^\chi) \text{ and } \nu \subseteq ((m-k)^k)
\right\}
\end{equation}
where
\begin{equation}
    \chi = \begin{cases}
        t & \nu_1 < m-k, \\
        t-1 & \nu_1 = m-k.
    \end{cases}
\end{equation}
The elementary symmetric polynomial $e_\lambda$ is in the full variable set $\{x_1,\dots,x_m\}$ while the Schur polynomial $s_\nu$ is in the last $k$ variables $\{x_{m-k+1},\dots,x_m\}$.
\end{definition}

When $t = 0$ one has $\LLL(m,k,0)=\{s_\nu(x_{m-k+1},\dots,x_m) \,:\, \nu \subseteq ((m-k-1)^k)\}$.
Roughly speaking, the polynomials in $\LLL(m,k,t)$  come from a single interval of $\mu_j = m$ indices arising from a partition $\mu \vdash n$; see Lemma~\ref{lem:A-equals-L}. The Schur polynomials $s_\nu(x_{m-k+1}, \dots,x_m)$ arising in Definition~\ref{defn:L-polynomials} should be compared with those in Lemmas~\ref{lem:abstract-leading} and~\ref{lem:additionbysatoshi}. 
We aim to show that  $\LLL(m,k,t)$ is linearly independent in $\FF[\xx_m]$. For this we use the following well-known result.

\begin{lemma}
    \label{lem:polynomial-ring-basis}
    Let $\BBB$ be any set of homogeneous polynomials in $\FF[\xx_m]$  which descends to a basis of the coinvariant ring $\FF[\xx_m]/I_m$. Then 
    \[ \{ e_1^{q_1} e_2^{q_2} \cdots e_m^{q_m} \times f \,:\, f \in \BBB, \, \, q_1, q_2,\dots,q_m \geq 0 \} \]
    is a basis of $\FF[\xx_m]$.
\end{lemma}

\begin{proof}
    Let $M$ be a graded $\FF[\xx_m]$-module, let $g \in \FF[\xx_m]_+$ be homogeneous of positive degree, and assume that $g$ is a non-zero divisor of $M$. If $B$ is a basis of $M/gM$ then $\{ g^a \cdot f \,:\, f \in B,\ a \geq 0 \}$ is a basis of $M$.  Since $I_m \subseteq \FF[\xx_m]$ is generated by the regular sequence $e_1, e_2,\dots,e_m$, the result follows.
\end{proof}

We are ready to show that the set $\LLL(m,k,t)$ is linearly independent in $\FF[\xx_m]$. We also show that if $f \in \LLL(m,k,t)$ and $x_1^{b_1} \cdots x_n^{b_m}$ appears in the monomial expansion of $f$, the exponent sequence $(b_1,\dots,b_m)$ satisfies a useful bound.

\begin{lemma}
\label{lem:polynomial-ring-linear-independence}
    Let $m,k,t \geq 0$ be such that $k \leq m$. 
    \begin{enumerate}
        \item If $x_1^{b_1} \cdots x_n^{b_m}$ is a monomial appearing in any polynomial in $\LLL(m,k,t)$ with nonzero coefficient, one has the componentwise inequality
        \[ (b_1,\dots,b_m) \leq (t,t+1,\dots,t+m-k-1,\overbrace{t+m-k-1, \dots, t+m-k-1}^{k}).\]
        \item The set $\LLL(m,k,t)$ of polynomials is linearly independent in $\FF[\xx_m]$.
    \end{enumerate}
\end{lemma}


The conditions in Lemma~\ref{lem:polynomial-ring-linear-independence} (1) are similar to the `sub-staircase' conditions defining the set $\AAA_m(J)$ of superspace Artin monomials for $J = \{m-k+1,\dots,m-1,m\},$ except that the staircase starts at $t+1$ instead of $1$. This `elevated staircase' arises naturally by considering various segments of $\stair(J(\mu,\gamma))$ for signed partitions $(\mu,\gamma)$. 

\begin{proof}
    For (1), consider an element $e_\lambda(x_1,\dots,x_m) \cdot s_{\nu}(x_{m-k+1},\dots,x_{m-1},x_m) \in \LLL(m,k,t).$ None of the exponents arising in the monomial expansion of $e_\lambda(x_1,\dots,x_m)$ are $ \leq \lambda'_1$. Furthermore, by considering semistandard tableaux we see that 
    \begin{quote}$s_\nu(x_{m-k+1},\dots,x_{m-1},x_m)$ is a linear combination of monomials $x_{m-k+1}^{c_{m-k+1}} \cdots x_{m-1}^{c_{m-1}} x_m^{c_m}$ for which $c_i \leq \nu_1$ for all $i$.
    \end{quote}
    Applying these facts to the product $e_\lambda(x_1,\dots,x_m) \times s_{\nu}(x_{m-k+1},\dots,x_{m-1},x_m)$ yields (1). 

    For (2), we claim that the set of Schur polynomials \begin{equation}
    \label{eq:restricted-schur}
    \{ s_\nu(x_{m-k+1},\dots,x_{m-1},x_m) \,:\, \nu_1 \leq m-k \text{ and } \nu'_1 \leq k \}\end{equation}
    is linearly independent in $R_m = \FF[\xx_m]/I_m$. 
    Indeed, every monomial in $s_\nu(x_{m-k+1},\dots,x_{m-1},x_m)$ lies in the Artin basis $\{ x_1^{a_1} \cdots x_m^{a_m} \,:\, a_i < i \}$ of $R_m$. We are therefore reduced to showing that \eqref{eq:restricted-schur} is linearly independent in $\FF[\xx_m]$, and this is well known. The linear independence assertion in (2) is a consequence of Lemma~\ref{lem:polynomial-ring-basis}.
\end{proof}

\begin{remark}
    \label{rmk:geometric-interpretation}
    The linear independence of \eqref{eq:restricted-schur} in $R_m$ has a geometric interpretation. Let $\Fl_m$ be variety of complete flags in $\CC^m$ with cohomology ring  $H^*(\Fl_m) = R_m$. The space $\Fl_m$ has an affine paving by open Schubert cells $C_w$ indexed by $w \in \symm_m$. The classes $[\overline{C_w}] \in H^*(\Fl_m)$ of these cells form a linear basis for $H^*(\Fl_m) = R_m$.  The set \eqref{eq:restricted-schur} corresponds to the representatives $[\overline{C_w}]$ for $w \in \symm_m$ with $w(1) < \cdots < w(m-k)$ and $w(m-k+1) < \cdots < w(m)$. See \cite{Fulton} for details.
\end{remark}

The astute reader may have observed that the proof of Lemma~\ref{lem:polynomial-ring-linear-independence} (1) actually shows the stronger componentwise inequality 
$(b_1,\dots,b_m) \leq (t^{m-k},(t+m-k-1)^k)$. The weaker statement in Lemma~\ref{lem:polynomial-ring-linear-independence} (1) is used for proving Lemma~\ref{lem:parabolic-dimension-equality} below.

We give an example of Lemma~\ref{lem:polynomial-ring-linear-independence}. Let $m = 5, k = 2,$ and $t=2$. Then
\[ (t,t+1,\dots,t+m-k-1,\overbrace{t+m-k-1,\dots,t+m-k-1}^k) = (2,3,4,4,4).\]
The set $\LLL(5,2,2)$ of polynomials in $\FF[\xx_5]$ may be written as a disjoint union
\begin{footnotesize}
\begin{multline*}
  \LLL(5,2,2) = \left\{ \begin{array}{ccccccc} e_{55}, & e_{54}, & e_{53}, & e_{52}, & e_{51}, & e_5, & e_{44}, \\
  e_{43}, & e_{42}, & e_{41}, & e_{4}, & e_{33}, & e_{32}, & e_{31}, \\ e_3, & e_{22}, & e_{21}, & e_2, & e_{11}, & e_1, & e_\varnothing\end{array} \right\} \cdot 
  \left\{ \begin{array}{cc} s_{22}(x_4,x_5) & s_{21}(x_4,x_5) \\ s_{2}(x_4,x_5) & s_{11}(x_4,x_5) \\ s_{1}(x_4,x_5) & s_{\varnothing}(x_4,x_5) \end{array} \right\} \sqcup \\
  \left\{ e_5, \, e_4, \, e_3, \, e_2, \, e_1,\, e_\varnothing\right\} \cdot 
  \left\{ \begin{array}{cc} s_{33}(x_4,x_5) & s_{32}(x_4,x_5) \\ s_{31}(x_4,x_5) & s_{3}(x_4,x_5)  \end{array} \right\}
\end{multline*}
\end{footnotesize}

\noindent
where the $e$'s are in the full variable set $\{x_1,\dots,x_5\}$. 
It is not hard to see that if $f \in \LLL(5,2,2)$ and  $m$ is a monomial appearing  in $f$, the exponent vector of $m$ is componentwise $\leq (2,3,4,4,4)$. Furthermore, the set of Schur polynomials $s_\mu(x_4,x_5)$ where $\mu \subseteq (3,3)$ descends to a linearly independent subset of $\FF[\xx_5]/I_5$. 

We want to relate the cardinality of $\LLL(m,k,t)$ to that of the set $\AAA(\mu,\gamma)$ introduced in Definition~\ref{def:antisymmetrized-artin}. To do this, we prove the following simple enumerative lemma.

\begin{lemma}
\label{lem:L-equals-I}
Let $m,k,t \geq 0$ be such that $k \leq n$. Let $\III(m,k,t)$ be the collection of nonnegative integer sequences $(c_1,\dots,c_m)$ of length $m$ such that $\dots$
\begin{itemize}
    \item we have $c_1 < \dots < c_{m-k}$,
    \item we have $c_{m-k+1} \leq \cdots \leq c_m$, and
    \item we have the componentwise inequality
    \[ (c_1,\dots,c_m) \leq (t,t+1,\dots,t+m-k-1,\overbrace{t+m-k-1, \dots, t+m-k-1}^{k}).\]
\end{itemize}
Then $\# \LLL(m,k,t) = \#\III(m,k,t)$.
\end{lemma}

\begin{proof}
We give a formula counting $\LLL(m,k,t)$. The idea is to consider the cases $\lambda \subseteq (m^{t-1})$ and $\lambda \supseteq (1^t)$ separately.
\begin{equation}
    \label{eq:L-count}
    \# \LLL(m,k,t) = \overbrace{{m+t-1 \choose m} \cdot {m \choose k}}^{\lambda \subseteq (m^{t-1})} + \overbrace{{m+t-1 \choose m-1} \cdot {m-1 \choose k}}^{\lambda \supseteq (1^t)}.
\end{equation}
In the first overbrace, the binomial coefficient ${m+t-1 \choose m}$ counts partitions $\lambda$ inside a $m \times (t-1)$ rectangle and the binomial coefficient  ${m \choose k}$ counts partitions $\nu$ inside a $(m-k) \times k$ rectangle.
In the second overbrace, the binomial coefficient ${m+t-1 \choose m-1}$ counts partitions $\lambda$ inside a $m \times t$ rectangle
which satisfy $\lambda \supseteq (1^t)$
while ${m-1 \choose k}$ counts partitions $\nu$ inside a $(m-k-1) \times k$ rectangle.

Sequences $(c_1,\dots,c_m) \in \III(m,k,t)$ are enumerated as follows.
{\small
\begin{equation}
    \label{eq:I-count}
    \# \III(m,k,t) = {m+t-k \choose t} \cdot {m+t-1 \choose k}=\left\{ {m+t-k-1 \choose t-1}+{m+t-k -1\choose t}\right\} \cdot {m+t-1 \choose m+t-k-1}.
\end{equation}
}

\noindent
Here ${m+t-k \choose t}$ counts the first part $0 \leq c_1 <\dots <c_{m-k} \leq t+m-k-1$ of the sequence while ${m+t-1 \choose k}$ counts the last part $0 \leq c_{m-k+1} \leq \dots \leq c_m \leq t+m-k-1$.
One checks that the quantities \eqref{eq:L-count} and \eqref{eq:I-count} coincide.
%
%
%
%
\end{proof}

Our next result is a direct relationship between $\# \AAA_n(\mu,\gamma)$ and the numbers $\# \LLL(-,-,-)$. This is the key enumerative result necessary to calculate the dimension of $\varepsilon_\mu \cdot SH_n$. 


\begin{lemma}
    \label{lem:A-equals-L}
    Let $(\mu,\gamma)$ be a signed partition of $n$ with $\mu = (\mu_1,\dots,\mu_p)$ and $\gamma = (\gamma_1,\dots,\gamma_p)$. Write
    \begin{multline}
        \stair(J(\mu,\gamma)) = \\(\underbrace{1,2,\dots,t_1,\overbrace{t_1,\dots,t_1}^{\gamma_1}}_{\mu_1} \mid  \underbrace{t_1 + 1, t_1 + 2, \dots, t_2, \overbrace{t_2, \dots, t_2}^{\gamma_2}}_{\mu_2} \mid  \dots \mid \underbrace{t_{p-1}+1,t_{p-1}+2,\dots,t_p,\overbrace{t_p,\dots,t_p}^{\gamma_p}}_{\mu_p.}).
    \end{multline}
    One has 
     \begin{equation}
     \label{eq:A-I-L}
        \# \AAA_n(\mu,\gamma) = \prod_{j=1}^p \# \III(\mu_j,\gamma_j,t_{j-1}) = 
        \prod_{j=1}^p \# \LLL(\mu_j,\gamma_j,t_{j-1}),
    \end{equation}
where $t_0=0$.
\end{lemma}


\begin{proof}
    The second equality in \eqref{eq:A-I-L} follows from Lemma~\ref{lem:L-equals-I}. The first equality is immediate from the definition of $\AAA_n(\mu,\gamma)$.
\end{proof}

\subsection{Lower bound on $\varepsilon_\mu \cdot SH_n$} We are ready to calculate the dimension of $\varepsilon_\mu \cdot SH_n$.

\begin{lemma}
    \label{lem:parabolic-dimension-equality}
    Let $\mu \vdash n$. The vector space $\varepsilon_\mu \cdot SH_n$ has dimension 
    \begin{equation}
        \dim_\FF \varepsilon_\mu \cdot SH_n = \sum_{\substack{\zero \leq \gamma \leq \mu}} \# \AAA_n(\mu,\gamma).
    \end{equation}
    The set 
    \begin{equation}
        \bigsqcup_{\zero \leq \gamma \leq \mu} \varepsilon_\mu \cdot ( \AAA_n(\mu,\gamma) \cdot \theta_{J(\mu,\gamma)})
    \end{equation}
    descends to a basis of $\varepsilon_\mu \cdot SR_n$.
\end{lemma}


\begin{proof}
Write $\mu = (\mu_1,\dots,\mu_p)$ and recall the modules $R_{\mu,\gamma}$ of Lemma~\ref{lem:additionbysatoshi}. Thanks to Lemmas~\ref{lem:parabolic-spanning} and \ref{lem:additionbysatoshi}, the lemma will follow if we can establish the inequality
    \begin{equation}
    \label{eq:desired-inequality}
        \dim_\FF R_{\mu,\gamma} \geq \# \AAA_n(\mu,\gamma).
    \end{equation}
Let $X_j=\{x_{\mu_1+\cdots+\mu_{j-1}+1},\dots,x_{\mu_1+\cdots+\mu_j}\}$ for $j=1,2,\dots,p$.
Let $\LLL(X_j,k,t)$ be the set of polynomials obtained from $\LLL(|X_j|=\mu_j,k,t)$ by replacing $x_1,x_2,\dots,$ with  $x_{\mu_1+\cdots+\mu_{j-1}+1},\dots,x_{\mu_1+\cdots+\mu_j}$.
Consider the set of polynomials
\[
\EEE(\mu,\gamma) =\{f_1\cdots f_p \,:\, f_j \in \LLL(X_j,\gamma_j,t_{j-1}) \mbox{ for }j=1,2,\dots,p\}\]
    where $t_0=0$ and $t_1,\dots,t_{p-1}$ are defined as in Lemma~\ref{lem:A-equals-L} by
    \begin{multline}
        \stair(J(\mu,\gamma)) = \\(\underbrace{1,2,\dots,t_1,\overbrace{t_1,\dots,t_1}^{\gamma_1}}_{\mu_1} \mid  \underbrace{t_1 + 1, t_1 + 2, \dots, t_2, \overbrace{t_2, \dots, t_2}^{\gamma_2}}_{\mu_2} \mid  \dots \mid \underbrace{t_{p-1}+1,t_{p-1}+2,\dots,t_p,\overbrace{t_p,\dots,t_p}^{\gamma_p}}_{\mu_p.}).
    \end{multline}
By Lemma~\ref{lem:A-equals-L},
the cardinality of $\EEE(\mu,\gamma)$ is exactly $\#\AAA_n(\mu,\gamma)$.
Also, elements of $\EEE(\mu,\gamma)$ are contained in $\FF[\xx_n]^{\symm_\mu} \cdot Z_{\mu,\gamma}$,
where $Z_{\mu,\gamma}$ is a generating set of $R_{\mu,\gamma}$ given in Lemma \ref{lem:additionbysatoshi}.
We claim that $\EEE(\mu,\gamma)$ is a set of linearly independent elements in $R_{\mu,\gamma}$.
    
    Theorem~\ref{thm:colon-ideal-basis} says that the set $\AAA_n(J(\mu,\gamma))$ of monomials which fit under the staircase $\stair(J(\mu,\gamma))$ descends to a basis of $\FF[\xx_n]/(I_n:f_{J(\mu,\gamma)}).$ Lemma~\ref{lem:polynomial-ring-linear-independence} (1) implies that, for any element $f$ in $\EEE(\mu,\gamma)$, every monomial appearing with nonzero coefficient in $f$ lies in $\AAA_n(J(\mu,\gamma))$. Therefore, to establish the claim we need only show that $\EEE(\mu,\gamma)$ is linearly independent in $\FF[\xx_n]$. This is a consequence of Lemma~\ref{lem:polynomial-ring-linear-independence} (2) and the claim is proven. 
\end{proof}

The proof says that the set $\EEE(\mu,\gamma)$ is actually an $\FF$-basis of $R_{\mu,\gamma}$.
When $\mu=(3,3,2)$ and $\gamma=(1,2,0)$,
we have
$\stair(J(\mu,\gamma))=(1,2,2 \mid 3,3,3 \mid 4,5)$,
and the set $\EEE(\mu,\gamma)$ is
{\Small
\[
\left\{
\begin{array}{ll}
s_1(x_3)\\
1\end{array}
\right\}
\cdot
\left\{
\begin{array}{ll}
e_3s_{11}(x_5,x_6),e_2s_{11}(x_5,x_6),e_1s_{11}(x_5,x_6),s_{11}(x_5,x_6),\\
e_3s_{1}(x_5,x_6),e_2s_{1}(x_5,x_6),e_1s_{1}(x_5,x_6),s_{1}(x_5,x_6),\\
e_{33}, e_{32},e_{31},e_{3}, e_{22},e_{21},e_{2}, e_{11},e_{1},1
\end{array}
\right\}
\cdot
\left\{
\begin{array}{ll}
e_\lambda(x_7,x_8): \lambda \subseteq (2,2,2)
\end{array}
\right\},
\]
}

\noindent
where $e_\lambda=e_\lambda(x_4,x_5,x_6)$ in the middle set. 
One may compare the cardinality of this set with that of $\AAA_n(\mu,\gamma)$ given after Definition \ref{def:antisymmetrized-artin}.

\begin{remark}
    \label{rmk:alternative-operator}
    Our proofs only used the easy direction of the Operator Theorem~\ref{thm:operator}: the space $SH_n$ contains $\delta_n$ and is closed under higher Euler operators.  Lemma~\ref{lem:parabolic-dimension-equality} in the case $\mu = (1^n)$ gives a  lower bound for the bigraded Hilbert series
    \[ \Hilb(SR_n;q,z) \geq \sum_{J \subseteq [n]} \Hilb(\FF[\xx_n]/(I_n:f_J);q) \cdot z^{n - \# J}.
    \]
    When combined with the upper bound on $\Hilb(SR_n;q,z)$ established in \cite[Sec. 3]{RW} (by easier means than the lower bound reasoning of \cite[Sec. 4]{RW}), this gives a new and simpler proof of the Operator Theorem and Fields Conjecture 1.
\end{remark}

\section{Main Results}
\label{sec:Main}

 We start with the ungraded $\symm_n$-structure of $SR_n$. Our proof is based on a simple fact: if $F, G \in \Lambda_n$ are degree $n$ symmetric functions, we have $F=G$ if and only if $e_\mu^\perp F = e_\mu^\perp G$ for all partitions $\mu \vdash n$. This holds because the Hall inner product is nondegenrate and $\{e_\mu \,:\, \mu \vdash n\}$ is a basis of $\Lambda_n$.

\begin{theorem}
    \label{thm:ungraded-fields} {\em (Fields Conjecture 2)}
    We have $SR_n \cong \FF[\OP_n] \otimes \sign$ as ungraded $\symm_n$-modules.
\end{theorem}

\begin{proof}
    Let $\mu \vdash n$. By Lemma~\ref{lem:module-antisymmetrization} and the fact mentioned above, it is enough to show
    \begin{equation}
    \label{eq:desired-dimensional-equality}
        \dim_\FF \varepsilon_\mu \cdot (SR_n) = \dim_\FF \varepsilon_\mu \cdot (\FF[\OP_n] \otimes \sign).
    \end{equation}
    Indeed, we have 
    \begin{equation}
        \dim_\FF \varepsilon_\mu \cdot (SR_n) = \sum_{\zero \leq \gamma \leq \mu} \# \AAA_n(\mu,\gamma) = \dim_\FF \varepsilon_\mu \cdot (\FF[\OP_n] \otimes \sign)
    \end{equation}
    where the first equality is a consequence of Lemma~\ref{lem:parabolic-dimension-equality} and the second equality follows from Lemma~\ref{lem:A-combinatorial-interpretation}.
\end{proof}

The bigraded Fields Conjecture can be proven according to the above strategy by showing 
\begin{equation}
\label{eq:parabolic-hilbert-skewing}
    \Hilb(\varepsilon_\mu \cdot SR_n;q,z) = e_\mu^\perp \left( \sum_{k=1}^n C_{n,k}(\xx;q) \cdot z^{n-k}\right)
\end{equation}
for all $\mu \vdash n$. While it is possible to prove \eqref{eq:parabolic-hilbert-skewing} directly from Lemma~\ref{lem:parabolic-dimension-equality} and known recursive formulas \cite[Lem. 3.7]{HRS} for $e_\mu^\perp C_{n,k}(\xx;q)$, we give a shorter proof using another module for $C_{n,k}(\xx;q)$.

\begin{theorem}
    \label{thm:fields} {\em (Fields Conjecture 3)}
    The bigraded $\symm_n$-module $SR_n$ has Frobenius characteristic
    \[ \grFrob(SR_n;q,z) = \sum_{k=1}^n C_{n,k}(\xx;q) \cdot z^{n-k}.\]
    where $q$ tracks bosonic degree and $z$ tracks fermionic degree.
\end{theorem}

\begin{proof}
    Fix a fermionic degree $r$ and let $k = n-r$. Consider the `superspace Vandermonde' $\delta_{n,k}$ defined in \cite{RW2}. This is the superspace element 
    \begin{equation}
    \label{eq:fields-one}
        \delta_{n,k} := \varepsilon_n \cdot (x_1^0 x_2^1 \cdots x_k^{k-1} x_{n-k+1}^{k-1} \cdots x_n^{k-1} \times \theta_{n-k+1} \cdots \theta_n)
    \end{equation}
    where $\varepsilon_n = \sum_{w \in \symm_n} \sign(w) \cdot w$.  Let $\ann_{\Omega_n}(\delta_{n,k}) := \{ f \in \Omega_n \,:\, f \odot \delta_{n,k}  = 0 \}$. The quotient ring
    \begin{equation}
    \label{eq:fields-two}
        \WWW_{n,k} := \Omega_n / \ann_{\Omega_n}(\delta_{n,k})
    \end{equation}
    is a bigraded $\symm_n$-module with top fermionic degree $r$. The module $\WWW_{n,k}$ was introduced in \cite{RW2} and further studied in \cite{RW3}. It is a consequence of \cite[Thm. 4.11]{RW3} that 
    \begin{equation} \bigsqcup_{\substack{J \subseteq [n] \\ |J| = r}} \AAA_n(J) \cdot \theta_J\end{equation}
    descends to a vector space basis of the top fermionic degree part of $\WWW_{n,k}$. Furthermore, it was proven in \cite{RW2} that 
    \begin{equation}
    \label{eq:fields-three}
        \text{coefficient of $z^{r}$ in } \grFrob(\WWW_{n,k};q,z) = C_{n,k}(\xx;q).
    \end{equation}
    Also, by e.g. \cite{HRW,HRS} we have
    \begin{equation}
    \label{eq:fields-four}
        \Frob( \FF[\OP_{n,k}] \otimes \sign) = C_{n,k}(\xx;1).
    \end{equation}
    Combining \eqref{eq:fields-three}, \eqref{eq:fields-four}, and Lemma~\ref{lem:A-combinatorial-interpretation} (and its proof) we see that 
    \begin{equation}
        \bigsqcup_{\substack{\zero \leq \gamma \leq \mu \\ \sum_j \gamma_j = r}} \{ \varepsilon_\mu \cdot (x_1^{a_1} \cdots x_n^{a_n} \times \theta_{J(\mu,\gamma)}) \,:\, x_1^{a_1}\cdots x_n^{a_n} \in \AAA_n(\mu,\gamma) \}
    \end{equation}
    descends to a basis of the top fermionic degree part of $\varepsilon_\mu \cdot \WWW_{n,k}$.  Therefore 
   \begin{align}
       \text{coefficient of $z^{r}$ in }\Hilb(\varepsilon_\mu \cdot SR_n;q,z) &= \sum_\gamma \sum_{m \in \AAA_n(\mu,\gamma)} q^{\deg(m)} \\
       &= \text{coefficient of $z^{r}$ in $\Hilb(\varepsilon_\mu \cdot \WWW_{n,k};q,z).$}
   \end{align}
   As in the proof of Theorem~\ref{thm:ungraded-fields}, this implies that the fermionic degree $r$ pieces of $SR_n$ and $\WWW_{n,k}$ are isomorphic as singly-graded $\symm_n$-modules. Since this is true for all $r$, the theorem is proved.
\end{proof}

Theorem~\ref{thm:fields} and \eqref{eq:fields-four} combine to give the module structure for the fermionic pieces of $SR_n$. This is a fermionic refinement of Theorem~\ref{thm:ungraded-fields}.

\begin{corollary}
    \label{cor:fermionic-piece}
    We have an isomorphism of ungraded $\symm_n$-modules $(SR_n)_{*,n-k} \cong \FF[\OP_{n,k}] \otimes \sign$.
\end{corollary}

The numbers $\# \OP_{n,k}$ satisfy the recursion
\begin{equation}
    \label{eq:op-recursion}
    \# \OP_{n,k} = k \cdot ( \# \OP_{n-1,k-1} + \# \OP_{n-1,k} ).
\end{equation}
Equation~\eqref{eq:op-recursion} may be established bijectively as follows: given $\sigma \in \OP_{n,k}$, erase the element $n$, together with its block if $\{n\}$ is a singleton. The following refinement of \eqref{eq:op-recursion} was conjectured by Reiner \cite{Reiner}.

\begin{corollary}
    \label{cor:reiner} {\em (Reiner's Conjecture)}
    The restriction $\Res^{\symm_n}_{\symm_{n-1}}(SR_n)_{*,n-k}$ satisfies  the recursion
    \begin{multline}
        \grFrob(\Res^{\symm_n}_{\symm_{n-1}}(SR_n)_{*,n-k};q) = \\
        [k]_q \cdot \left(\grFrob((SR_{n-1})_{*,n-k};q) + \grFrob((SR_{n-1})_{*,n-k-1};q)\right).
    \end{multline}
\end{corollary}

\begin{proof}
    For any $\symm_n$-module $V$, one has $\Frob(\Res^{\symm_n}_{\symm_{n-1}} V) = e_1^\perp \Frob(V)$. By Theorem~\ref{thm:fields} and Corollary~\ref{cor:fermionic-piece} we are reduced to showing 
    \begin{equation}
    \label{eq:c-skewing-by-1}
        e_1^\perp C_{n,k}(\xx;,q) = [k]_q \cdot ( C_{n-1,k-1}(\xx;q) + C_{n-1,k}(\xx;q)).
    \end{equation}
    Equation~\eqref{eq:c-skewing-by-1} follows from \cite[Lem 3.13, $j=1$]{RW3}.
\end{proof}

\section{Springer theory}
\label{sec:Springer}


While various geometric interpretations the symmetric function $C_{n,k}(\xx;q)$ representing individual fermionic pieces of $SR_n$ are available \cite{GLW,GGG,HRS2,PR}, Springer theory leads to an interpretation of $SR_n$ itself which applies to arbitrary Lie type. We start with some  coinvariant notation.

Let $W$ be a reflection group acting irreducibly on its reflection representation $V$. We write $S^{(k \mid j)}$ for the $\ZZ_{\geq 0}^k \times \ZZ_{\geq 0}^j$-graded algebra 
\[ S^{(k \mid j)} := \FF[\overbrace{V \oplus \cdots \oplus V}^k] \otimes \wedge ( \overbrace{V \oplus \cdots \oplus V}^j ).\]
If $\dim V = n$ then $S^{(1 \mid 1)} = \Omega_n$ is the superspace ring. The natural action of $W$ on $V$ induces an action of $W$ on $S^{(k \mid j)}$. We write $(S^{(k \mid j)})^W_+ \subseteq S^{(k \mid j)}$ for the $W$-invariants with vanishing constant term. We define the {\em generalized coinvariant ring} by
\[ R^{(k \mid j)}_W := S^{(k \mid j)}/ I_W \quad \quad \text{where $I_W \subseteq S^{(k \mid j)}$ is the ideal generated by $(S^{(k \mid j)})^W_+$.}\]
In Proposition~\ref{prop:SR-geometric} below we use a result of Reeder \cite{Reeder} to give a geometric interpretation of $R^{(1 \mid 1)}_W$ for any Weyl group $W$.

We apply Springer theory in the context of compact Lie groups; see \cite{Reeder}. Let $G$ be a compact simple Lie group and let $T \subseteq G$ be a maximal torus. Write $W := N_G(T)/T$ for the Weyl group of $G$. Then $W$ acts on $T$ by conjugation. The space $G/T$ is homotopy-equivalent to the flag variety associated to $G$ and carries an action of $W$ given by $w \cdot gT = g n^{-1} T$ where $w = nT$. Define $\widetilde{G} := G/T \times T$. Then $W$ acts diagonally  on $\widetilde{G}$. We have a map
\begin{equation}
\label{eq:p-map}
    p: \widetilde{G} \longrightarrow G, \quad\quad p: (gT,t) \mapsto g t g^{-1}.
\end{equation}
The map $p$ is a homotopy model for the usual Grothendieck--Springer resolution ${\mathsf G} \times^{\mathsf B} {\mathsf B} \to {\mathsf G}$ obtained by intersection with a unitary group.

Let $H^*(-)$ denote cohomology with coefficients in $\FF$. Let $V := \tttt$ be the Lie algebra of $T$ with its natural action of $W$. Borel's Theorem \cite{Borel} identifies $H^*(G/T) = R^{(1 \mid 0)}_W$. Furthermore, the torus $T$ has 
$H^*(T) \cong \wedge V = R^{(0 \mid 1)}_W$
where we used the result of Steinberg (see \cite[Thm. A, \S 24-3, p. 250]{Kane}) that $\wedge^i V$ is a nontrivial irreducible $W$-module for all $0 < i \leq \dim V$.
The K\"unneth Theorem implies
\begin{equation}
    H^*(\widetilde{G}) = H^*(G/T \times T) = H^*(G/T) \otimes H^*(T) = R^{(1\mid 0)}_W \otimes R^{(0\mid 1)}_W.
\end{equation}
The morphism $p: \widetilde{G} \to G$ induces a map of cohomology rings
\begin{equation}
    p^*: H^*(G) \longrightarrow H^*(\widetilde{G}).
\end{equation}

\begin{theorem}
    \label{thm:reeder} {\em (Reeder \cite[Prop. 6.1]{Reeder})}
   The map $p^*$ is injective with image $H^*(\widetilde{G})^W$.
\end{theorem}

Reeder's Theorem~\ref{thm:reeder} gives the following geometric interpretation of $R^{(1 \mid 1)}_W$.

\begin{proposition}
    \label{prop:SR-geometric}
    Write $H^+(G) := \bigoplus_{i > 0} H^i(G)$ for the positive degree part of the cohomology of $G$ and let \begin{equation} J := \left( p^*(H^+(G)) \right) \subseteq H^*(\widetilde{G})\end{equation} be the ideal $J \subseteq H^*(\widetilde{G})$ generated by the image of $H^+(G)$ under $p^*$. There holds an isomorphism
    \begin{equation}
        R^{(1 \mid 1)}_W \cong H^*(\widetilde{G})/J
    \end{equation}
    of bigraded $W$-algebras.
\end{proposition}

Our derivation of Proposition~\ref{prop:SR-geometric} from Theorem~\ref{thm:reeder} uses the ubiquitous representation theoretic technique of averaging over $W$. This is permissible since $\FF$ has characteristic 0.

\begin{proof}
    Let $A = \bigoplus_{d \geq 0} A_d$ and $B = \bigoplus_{d \geq 0} B_d$ be graded commutative or anticommutative $\FF$-algebras which carry graded actions of $W$. Then $A \otimes B$ inherits a grading $(A \otimes B)_d = \oplus_{i+j=d} A_i \otimes B_j$ and a graded $W$-action $w \cdot (a \otimes b) := (w \cdot a) \otimes (w \cdot b).$
    
    Let $A^W_+ \subseteq A, B^W_+ \subseteq B, (A \otimes B)^W_+ \subseteq A \otimes B$ be the spaces of $W$-invariants with vanishing constant term. Let $I_A = (A^W_+), I_B = (B^W_+), I_{A \otimes B} = ((A \otimes B)^W_+)$ be the ideals generated by these spaces. There is a natural surjection of graded algebras
    \begin{equation}
        \pi: A \otimes B \twoheadrightarrow A/I_A \otimes B/I_B.
    \end{equation}
    Let $K \subseteq (A/I_A \otimes B/I_B)$ be the ideal generated by $W$-invariants in $A/I_A \otimes B/I_B$ with vanishing constant term. We claim that $\pi^{-1}(K) = I_{A \otimes B}$. 

    The containment $I_{A \otimes B} \subseteq \pi^{-1}(K)$ is may be checked on generators of $I_{A \otimes B}$. On the other hand, since $\FF$ has characteristic 0, the ideal $K \subseteq A/I_A \otimes B/I_B$ is generated by elements of the form
    \begin{equation}
        f_{a,b} := \sum_{w \in W} w \cdot ((a + I_A) \otimes (b + I_B)) \quad \text{$a \in A, b \in B$ homogeneous with $\deg(a) + \deg(b) > 0.$}
    \end{equation}
    If $h \in \pi^{-1}(K)$ there exists a finite expression
    \begin{equation}
        \pi(h) = \sum_{a,b} \left( \sum_i (q_{a,b,i} + I_A) \otimes (p_{a,b,i} + I_B)\right) \cdot f_{a,b}
    \end{equation}
    for some $q_{a,b,i} \in A$ and $p_{a,b,i} \in B$. One has 
    \begin{equation}
        h' := \sum_{a,b} \left( \sum_i q_{a,b,i} \otimes p_{a,b,i} \right) \cdot \left( \sum_{w \in W} w \cdot (a \otimes b) \right) \in I_{A \otimes B}
    \end{equation}
    and $\pi(h) = \pi(h')$ so that 
    \begin{equation}
        h - h' \in \Ker(\pi) = I_A \otimes B + A \otimes I_B \subseteq I_{A \otimes B}.
    \end{equation}
    It follows that $h \in I_{A \otimes B}$. The map $\pi$ therefore gives rise to an isomorphism
    \begin{equation}
        \label{eq:abstract-isomorphism}
        (A \otimes B)/I_{A \otimes B} \cong (A/I_A \otimes B/I_B)/K.
    \end{equation}
    To prove the proposition, we take $A = \FF[V]$ and $B = \wedge V$. The left-hand side of \eqref{eq:abstract-isomorphism} becomes $R^{(1 \mid 1)}_W$ while Theorem~\ref{thm:reeder} identifies the right-hand side of \eqref{eq:abstract-isomorphism} with $H^*(\widetilde{G})/J$.
\end{proof}

In type A, our results on $SR_n$ combine with Proposition~\ref{prop:SR-geometric} as follows. Let $G = SU(n)$ and let $T \subseteq G$ be the diagonal subgroup.
Write $H^\Top(X)$ for the top-degree nonvanishing cohomology of a space $X$.  Recall that a {\em composition} $\alpha$ of $n$ is a list $\alpha = (\alpha_1,\dots,\alpha_k)$ of positive integers which sum to $n$. For any such $\alpha$, choose $t_\alpha \in T$ whose $\symm_n$ stabilizer is $\symm_\alpha = \symm_{\alpha_1} \times \cdots \times \symm_{\alpha_k}$. The fibers $p^{-1}(t_\alpha) \subseteq G/T \times T$ are related to $SR_n$ as follows.

\begin{corollary}
    \label{cor:fiber} Let $J \subseteq H^*(G/T \times T)$ be the ideal generated by $p^*(H^+(G))$. We have isomorphisms of ungraded $\symm_n$-modules
    \begin{equation}
        H^*(G/T \times T)/J \cong SR_n \cong \FF[\OP_n] \otimes \sign \cong   \bigoplus_{\alpha} H^\Top(p^{-1}(t_\alpha))
    \end{equation}
    where the direct sum is over all compositions $\alpha$ of $n$.
\end{corollary}

\begin{proof}
    The isomorphisms $H^*(G/T \times T)/J \cong SR_n \cong \FF[\OP_n] \otimes \sign$ follow from Theorem~\ref{thm:ungraded-fields} and Proposition~\ref{prop:SR-geometric}. We establish the final isomorphism $\FF[\OP_n] \otimes \sign \cong \bigoplus_\alpha H^\Top(p^{-1}(t_\alpha))$ as follows.

    For any composition $\alpha = (\alpha_1,\dots,\alpha_k)$ of $n$, let $L_\alpha = SU(\alpha_1) \times \cdots \times SU(\alpha_k) \subseteq G$ be the Levi subgroup associated to $\alpha$. Let $T_\alpha := T \cap L_\alpha$. The K\"unneth Theorem and Borel's presentation of $H^*(\Fl(n))$ imply 
    \begin{equation}
        H^*(L_\alpha/T_\alpha) \cong H^*(\Fl(\alpha_1)) \otimes \cdots \otimes H^*(\Fl(\alpha_k)) \cong R_{\alpha_1} \otimes \cdots \otimes R_{\alpha_k}.
    \end{equation}
    In particular, the top-degree part $H^\Top(L_\alpha/T_\alpha) \cong \sign_{\symm_\alpha}$ carries the sign representation of $\symm_\alpha$. On the other hand, one calculates that the fiber $p^{-1}(t_\alpha)$ is a disjoint union
    \begin{equation}
        p^{-1}(t_\alpha) = \bigsqcup_{w \symm_{\alpha} \in \symm_n/\symm_{\alpha}} X_\alpha
    \end{equation}
    over left cosets $w \symm_{\alpha} \in \symm_n/\symm_\alpha$ where each $X_\alpha$ is a copy of $L_\alpha/T_\alpha$. One has the induction product
    \begin{equation}
        H^\Top(p^{-1}(t_\alpha)) \cong \Ind_{\symm_{\alpha}}^{\symm_n} H^\Top(L_\alpha/T_\alpha) \cong \Ind_{\symm_\alpha}^{\symm_n}(\sign_{\symm_\alpha}) \cong \FF[\OP_\alpha] \otimes \sign
    \end{equation}
    where $\OP_\alpha = \{ (B_1 \mid \cdots \mid B_k) \in \OP_n \,:\, \# B_i = \alpha_i \text{ for all $i$} \}$. The desired $\symm_n$-module isomorphism $\bigoplus_{\alpha} H^\Top(p^{-1}(t_\alpha)) \cong \FF[\OP_n] \otimes \sign$ follows by taking a direct sum over all $\alpha$.
\end{proof}


It would be interesting to have a purely geometric proof of Corollary~\ref{cor:fiber}.


\section{Conclusion}
\label{sec:Conclusion}

In this paper we proved that the module structure of $R^{(1 \mid 1)}_W$ is governed by ordered set partitions when $W$ is of type A. We propose a combinatorial model which approximates the quotient ring $R^{(1\mid 1)}_W$ in general.
Let $\mathbb{S} \subset V$ be a sphere centered at the origin. The reflecting hyperplanes of $W$ subdivide $\mathbb{S}$ into faces of various dimensions resulting in the {\em Coxeter complex} $\OP_W$. The $k$-dimensional faces $\OP_{W,k} \subseteq \OP_W$ carry an action of $W$ for each $k$.

The fermionic Hilbert series $\Hilb(R^{(1 \mid 1)}_W;1,z)$ has coefficient sequence given by the reversed $f$-vector of $\OP_W$ in many cases. In type A this follows from Fields Conjecture 1. In forthcoming work, Bhattacharya proves \cite{Bhattacharya} this coincidence in types BC. However, in type F$_4$ one has $$\Hilb(R^{(1|1)}_W;1,q) = 1152+ 2304 z + \mathbf{1396} z^2 + \mathbf{244} z^3 + z^4$$ while the reversed $f$-vector of $\OP_W$, $(1152,2304,\mathbf{1392},\mathbf{240},1)$, is slightly smaller. We conjecture an equivariant version of this phenomenon.

\begin{conjecture}
    \label{conj:surjection}
    Let $\RR[\OP_W]$ be the permutation module afforded by the Coxeter complex $\OP_W$. There exists a $W$-equivariant surjection
    \[ R^{(1\mid1)}_W \twoheadrightarrow \RR[\OP_W] \otimes \det\]
    where $\det$ is the 1-dimensional determinant representation of $W$. In fact, for any $k$, there is a $W$-equivariant surjection \[ (R^{(1 \mid 1)}_W)_{*,n-k} \twoheadrightarrow \RR[\OP_{W,k}] \otimes \det.\]
\end{conjecture}

Corollary~\ref{cor:fermionic-piece} proves Conjecture~\ref{conj:surjection} in type A.
The type F$_4$ Weyl group $W$ has 25 irreducible characters; we give them the same order $\chi_1,\dots,\chi_{25}$ as in {\tt sage}. If $\chi$ is any character of $W(F_4)$, we obtain an expansion $\chi = a_1 \chi_1 + \cdots + a_{25} \chi_{25}$ for $a_i \geq 0$. The following table gives the sequences $(a_1,\dots,a_{25})$ for $k = 1,2,3$ with   $R^{(1|1)}_W$ on top and $\RR[\OP_W] \otimes \det$ on the bottom.
\begin{center}
\begin{tabular}{c | l}
$k$ & \text{characters of $(R_{W(F_4)}^{(1|1)})_{*,n-k}$ and $\RR[\OP_{W(F_4),k}] \otimes \det$} \\ \hline
1 & $
\begin{array}{l}
    (0, 4, 0, 0, 2, 0, 2, 0, 0, 1, 1, 4, {\bf 1}, 0, 2, 0, 4, 0, 4, 6, 1, 1, 0, 2, 2)\\
    (0, 4, 0, 0, 2, 0, 2, 0, 0, 1, 1, 4, {\bf 0}, 0, 2, 0, 4, 0, 4, 6, 1, 1, 0, 2, 2)
\end{array}$ \\ \hline
2 & $\begin{array}{l}
    (0, 6, 1, 1, 6, 0, 6, 0, 0, 5, 5, 12, {\bf 5}, 6, 8, 4, 16, 4, 16, 21, 10, 10, 3, 14, 18)\\
    (0, 6, 1, 1, 6, 0, 6, 0, 0, 5, 5, 12, {\bf 4}, 6, 8, 4, 16, 4, 16, 21, 10, 10, 3, 14, 18)
\end{array}$ \\ \hline
3 & $\begin{array}{l}
    (0, 4, 2, 2, 6, 2, 6, 2, 4, 8, 8, 12, 8, 12, 12, 12, 20, 12, 20, 24, 18, 18, 12, 24, 32)\\
    (0, 4, 2, 2, 6, 2, 6, 2, 4, 8, 8, 12, 8, 12, 12, 12, 20, 12, 20, 24, 18, 18, 12, 24, 32)
\end{array} $
\end{tabular}
\end{center}
The cases $k = 1,2$ disagree by one copy of the same  {\bf bolded} unlucky degree 4 character $\chi_{13}$ with $(R^{(1|1)}_W)_{*,n-k}$ being larger in both cases. We have $(R^{(1|1)}_W)_{*,4} \cong \RR[\OP_{W,0}] \otimes \det \cong \det$.

We compare Conjecture~\ref{conj:surjection} with other results on $R^{(k \mid j)}_W$ for an irreducible reflection group $W$.
\begin{itemize}
    \item As in the discussion before Theorem~\ref{thm:reeder}, one has $R^{(0 \mid 1)}_W = \wedge V$ for all $W$.
    \item Chevalley proved \cite{Chevalley} that $R^{(1 \mid 0)}_W$ carries the regular representation $\RR[W]$ for all $W$.
    \item Let $W$ be a Weyl group with Coxeter number $h$ and root lattice $Q$. The {\em finite torus} $Q/(h+1)Q$  carries an action of $W$. Gordon established \cite{Gordon} a $W$-equivariant surjection
    \begin{equation}
    \label{eq:gordon-surjection}
    R^{(2 \mid 0)}_W \twoheadrightarrow \RR[Q/(h+1)Q] \otimes \det.\end{equation}
    By Haiman's work \cite{Haiman}, the surjection \eqref{eq:gordon-surjection} is an isomorphism in type A. It is not an isomorphism in the other classical types BCD.
    \item For all $W$, Kim--Rhoades proved \cite{KR} 
    \begin{equation}
        \label{eq:kim-rhoades}
        [(R^{(0 \mid 2)}_W)_{i,j}] = [ \wedge^i V] \cdot [\wedge^j V]  - [\wedge^{i-1} V] \cdot [\wedge^{j-1} V] \quad \quad \text{for } i+j \leq \dim_\FF V
    \end{equation}
   in the Grothendieck ring of $W$. If $i + j > \dim_\FF V$ one has $(R^{(0 \mid 2)}_W)_{i,j} = 0$.
\end{itemize}
Examining these cases, the combinatorial approximations of $R^{(k\mid j)}_W$ become looser and less uniform as $k,j$ grow. The behavior of $R^{(2 \mid 0)}_W, R^{(1 \mid 1)}_W,$ and $R^{(0\mid2)}_W$ suggests that fermionic sets of variables may admit easier and more faithful combinatorial models than their bosonic counterparts.
Gordon's result \cite{Gordon} was proven by indirect means using the {\em rational Cherednik algebra} $\HHH_c$ attached to $W$. This is a filtered deformation of the smash product $S^{(2 \mid 0)} * W$ which depends on a parameter $c$. The associated graded algebra of $\HHH_c$ recovers $S^{(2 \mid 0)} * W$. Conjecture~\ref{conj:surjection} may be best proven via an appropriate deformation of $S^{(1 \mid 1)} * W$.

\section*{Acknowledgements}

The authors are grateful to Tom Braden, Vic Reiner, and Minh-Tam Trinh for very helpful conversations. S.\ Murai was partially supported by KAKENHI 25K06943 and 23K17298.
B.\ Rhoades was partially supported by NSF Grant DMS-2246846. A.\ Wilson was partially supported by AMS-Simons PUI Grant 434651.

\end{document}